\documentclass[twoside]{article}
\usepackage[utf8]{inputenc}
\usepackage[T2A]{fontenc}
\usepackage{array}
\usepackage{amssymb}
\usepackage{amsmath}
\usepackage{amsthm}
\usepackage{latexsym}
\usepackage{hyperref}
\usepackage{indentfirst}
\usepackage{appendix}
\usepackage{bm}
\usepackage{enumerate}
\usepackage[dvips]{graphicx}
\usepackage{epsf}
\usepackage{wrapfig}
\usepackage[mathscr]{euscript}
\usepackage{indentfirst}
\usepackage[english,russian]{babel}
\usepackage{textcomp}
\usepackage{diagrams}
\newtheorem{theorem}{Theorem}
\newtheorem{lemma}{Lemma}
\newtheorem{definition}{Definition}
\newtheorem{proposition}{Proposition}

\newtheorem{corollary}{Corollary}
\newtheorem{remark}{Remark}

\newcommand{\boxto}{\ensuremath{%
		\mathrel{\Box\kern-1.5pt\raise1pt\hbox{$\mathord{\rightarrow}$}}}}
	
\newcommand{\diamondto}{\ensuremath{%
		\mathrel{\Diamond\kern-1.5pt\raise1pt\hbox{$\mathord{\rightarrow}$}}}}
\newcommand{\boxTo}{\ensuremath{%
		\mathrel{\Box\kern-1.5pt\raise1pt\hbox{$\mathord{\Rightarrow}$}}}}
\newcommand{\diamondTo}{\ensuremath{%
		\mathrel{\Diamond\kern-1.5pt\raise1pt\hbox{$\mathord{\Rightarrow}$}}}}
\newcommand*\vDdash{%
	\mathrel{%
		\ooalign{$\vdash$\cr$\vDash$}%
	}%
}
\def\@clipped@vdash{%
	\raise .6ex\hbox{\clipbox{0pt .6ex 0pt .6ex}{$\vdash$}}%
}

\begin{document}
{\selectlanguage{english}
\binoppenalty = 10000 %
\relpenalty   = 10000 %

\pagestyle{headings} \makeatletter
\renewcommand{\@evenhead}{\raisebox{0pt}[\headheight][0pt]{\vbox{\hbox to\textwidth{\thepage\hfill \strut {\small Grigory K. Olkhovikov}}\hrule}}}
\renewcommand{\@oddhead}{\raisebox{0pt}[\headheight][0pt]{\vbox{\hbox to\textwidth{{A basic Nelsonian conditional logic}\hfill \strut\thepage}\hrule}}}
\makeatother

\title{A basic system of paraconsistent Nelsonian logic of conditionals}
\author{Grigory K. Olkhovikov\\ Department of Philosophy I\\ Ruhr University Bochum\\
email: grigory.olkhovikov@\{rub.de, gmail.com\}}
\date{}
\maketitle
\begin{quote}
{\bf Abstract.} We define a Kripke semantics for a conditional logic based on the propositional logic $\mathsf{N4}$, the paraconsistent variant  of Nelson's logic of strong negation; we axiomatize the minimal system induced by this semantics. The resulting logic, which we call $\mathsf{N4CK}$, shows strong connections both with the basic intuitionistic logic of conditionals $\mathsf{IntCK}$ introduced earlier in \cite{olkhovikov} and with the $\mathsf{N4}$-based modal logic $\mathsf{FSK}^d$ introduced in \cite{odintsovwansing} as one of the possible counterparts to the classical modal system $\mathsf{K}$. We map these connections by looking into the embeddings which obtain between the aforementioned systems.
\end{quote}
\begin{quote}{\bf Keywords.} conditional logic, strong negation, paraconsistent logic, strong completeness, modal logic, constructive logic 
\end{quote}

\section{Introduction}
The present paper is devoted to a study of a new system of conditional logic (called $\mathsf{N4CK}$) which conservatively extends the logic $\mathsf{N4}$, the paraconsistent variant of Nelson's logic of strong negation.\footnote{The only difference between $\mathsf{N4}$ and $\mathsf{N3}$, the original version of Nelson's logic of strong negation, is that in $\mathsf{N4}$ the extensions and anti-extensions of propositional letters are no longer required to be disjoint.}  In our opinion, $\mathsf{N4CK}$ can be correctly viewed as the minimal `normal' $\mathsf{N4}$-based system of conditional logic; in this capacity, $\mathsf{N4CK}$ is the right counterpart to such systems as the classical conditional logic $\mathsf{CK}$ and the intuitionistic conditional logic $\mathsf{IntCK}$, introduced in \cite{chellas} and \cite{olkhovikov}, respectively.

The rest of the paper is organized as follows. Section \ref{S:Prel} introduces the notational preliminaries, including the set of notations for different logical systems to be mentioned in the paper. This is followed by Section \ref{S:basis} where we introduce the logic $\mathsf{N4}$, which forms the purely propositional basis of $\mathsf{N4CK}$ together with its closest relatives like the intuitionistic logic $\mathsf{IL}$ and the classical logic $\mathsf{CL}$. We briefly survey the properties of $\mathsf{N4}$, which serves both to make the paper more self-contained and to prepare the reader for the study of $\mathsf{N4CK}$ as many of the properties stated in Section \ref{S:Prel} for $\mathsf{N4}$ will also be demonstrated for $\mathsf{N4CK}$ in the later sections. After that, Section   \ref{S:system} introduces $\mathsf{N4CK}$, the main subject of the paper. More precisely, the system is introduced in Section \ref{sub:lands} by means of its intended Kripke semantics, whereas in Section \ref{sub:axiomatization} we axiomatize the logic. Section \ref{S:other} explores the relation of $\mathsf{N4CK}$ to several other logics; it is sub-divided into Section \ref{sub:conditional} and Section \ref{sub:modal}, where in the former we show that (1) $\mathsf{N4CK}$ yields $\mathsf{CK}$ as soon as its propositional basis is extended from $\mathsf{N4}$ to $\mathsf{CL}$, and (2) that $\mathsf{N4CK}$ is embeddable into the positive fragment of $\mathsf{IntCK}$; the embedding in question extends the well-known embedding of $\mathsf{N4}$ into the positive fragment of  $\mathsf{IL}$. In Section \ref{sub:modal} we argue that the modal counterpart of $\mathsf{N4CK}$ is provided by the $\mathsf{N4}$-based modal logic $\mathsf{FSK}^d$ introduced in \cite{odintsovwansing} and compare the relation between the two logics with the relation that obtains between $\mathsf{IntCK}$ and its modal companion $\mathsf{IK}$.

Finally, in Section \ref{S:conclusion}, we briefly discuss the results of the previous sections, and, after drawing some conclusions, describe several avenues for continuing the research lines presented in the paper. The paper also has several appendices where the reader can find the more technical parts of our reasoning which we include for the sake of completeness.

\section{Preliminaries}\label{S:Prel}
We use this section to fix some notations to be used throughout the paper.

We will use IH as the abbreviation for Induction Hypothesis in the inductive proofs; $\alpha:=\beta$ means that we define $\alpha$ as $\beta$. We will use the usual notations for sets and functions. As for the sets, we will write $X\Subset Y$, iff $X\subseteq Y$ and $X$ is finite. We will understand the natural numbers as the finite von Neumann ordinals; we denote their set by $\omega$. We will extensively use ordered couples of sets which we will also call \textit{bi-sets}. The usual set-theoretic relations and operations on bi-sets will be understood componentwise, so that, e.g. $(X,Y)\subseteq(Z,W)$ means that $X \subseteq Z$ and $Y \subseteq Y$ and similarly in other cases.

Relations will be understood as sets of ordered tuples where the length of the tuple defines the \textit{arity} of the relation. Given binary relations $R \subseteq X \times Y$ and $S\subseteq Y\times Z$, we set $R\circ S:= \{(a,c)\mid\text{ for some }b \in Y,\,(a, b)\in R\,\&\,(b,c)\in S\}$, and $R^{-1}:=\{(b,a)\mid (a,b)\in R\}$.

Functions will be understood as relations with special properties; we will write $f:X \to Y$ to denote a function $f \subseteq X\times Y$ such that its left projection is all of $X$. If  $f: X \to Y$ and $Z\subseteq X$ then we will denote the image of $Z$ under $f$ by $f(Z)$. In view of our previous convention for relations, for any given two functions $f:X\to Y$ and $g:Y\to Z$, we will denote the function $x\mapsto g(f(x))$ by $f\circ g$, even though, in the existing literature, this function is often denoted by $g\circ f$ instead. 

In this paper we will compare several conditional and modal logics to one another, therefore, we will need a notation for languages allowing us to easily switch between different variants of such logics. In general, if $\Omega$ is a set of propositional letters and $f^{i_1}_1,\ldots, f^{i_n}_n$ are connectives, where $n, i_1,\ldots, i_n \in \omega$ and $i_1,\ldots, i_n$ indicate the arities, we will denote by $\mathbb{L}(\Omega, f^{i_1}_1,\ldots, f^{i_n}_n)$ the language that is the smallest set of formulas extending $\Omega$ and closed under the applications of connectives in $\{f^{i_1}_1,\ldots, f^{i_n}_n\}$.

However, within this paper we will confine ourselves to a rather narrow range of instantiations of this general construction. First of all, let $Var:=\{p_i\mid i\in \omega\}$ and $Var':= \{q_i\mid i\in \omega\}$ be disjoint. In this paper, we will not consider any instantiations of $\Omega$ outside the set $\{Var, Var\cup Var'\}$. Furthermore, the main logic that we will consider in this paper, the basic Nelsonian logic of conditionals $\mathsf{N4CK}$, will be seen to conservatively extend the logic $\mathsf{N4}$, the paraconsistent variant of Nelson's logic of strong negation. Therefore, we will mostly consider languages of the form $\mathbb{L}(\Omega, f^{i_1}_1,\ldots, f^{i_n}_n)$ where $\{f^{i_1}_1,\ldots, f^{i_n}_n\}\supseteq \{\wedge^2, \vee^2, \sim^1, \to^2\}$; our basic language will be $\mathcal{L}:= \mathbb{L}(Var, \wedge^2, \vee^2, \sim^1, \to^2)$. We will want to extend $\mathcal{L}$ with some connectives from the set $\{\boxto^2, \diamondto^2, \Box^1, \Diamond^1\}$ putting these additional connectives into the subscript and omitting their arities, so that, for example, $\mathcal{L}_{(\Box,\Diamond)}$ will mean the language $ \mathbb{L}(Var, \wedge^2, \vee^2, \sim^1, \to^2, \Box^1, \Diamond^1)$.

We will sometimes want to omit $\sim^1$ from the set of connectives thus getting the \textit{positive} version of the language, or extend the set of propositional letters from $Var$ to $Var \cup Var'$ thus getting the \textit{extended} version of the language, or both. The first two operations will be expressed by appending the superscripts $+$ and $e$, respectively to the language notation so that, e.g. $\mathcal{L}^{e+}_{\boxto}$ will mean $\mathbb{L}(Var\cup Var', \wedge^2, \vee^2, \to^2, \boxto^2)$. Thus the minimal language that we will consider in this paper is going to be $\mathcal{L}^+$ and every other language mentioned below will extend it. The elements of languages will be called their \textit{formulas}.

Although one and the same logic can often be formulated over different languages, in this paper we will abstract away from such subtleties, and will simply treat a logic as a set $\mathsf{L} \subseteq \mathcal{P}(Lang)\times\mathcal{P}(Lang)$, for some language $Lang$, where $(\Gamma, \Delta)\in \mathsf{L}$ iff $\Delta$ $\mathsf{L}$-follows from $\Gamma$ (we will also denote this by $\Gamma\models_\mathsf{L}\Delta$). We will say that $(\Gamma, \Delta)$ is $\mathsf{L}$-satisfiable iff $(\Gamma, \Delta)\notin \mathsf{L}$. Given a $\phi \in Lang$, $\phi$ is $\mathsf{L}$-valid (we will also write $\phi \in \mathsf{L}$) iff $(\emptyset, \{\phi\}) \in \mathsf{L}$ and $\phi$ is $\mathsf{L}$-satisfiable iff $(\{\phi\},\emptyset) \notin \mathsf{L}$.

Since every language considered in this paper contains $\vee$, we can define the Disjunction Property (DP) for an arbitrary logic already at this introductory stage. We will say that a logic $\mathsf{L}\subseteq \mathcal{P}(Lang)\times\mathcal{P}(Lang)$ has DP iff for all $\phi, \psi \in Lang$, we have $\phi \vee \psi \in \mathsf{L}$ iff at least one of $\phi, \psi$ is in $\mathsf{L}$. In case $Lang$ also includes $\sim$, we can also speak about the Constructible Falsity Property (CFP). A logic $\mathsf{L}\subseteq \mathcal{P}(Lang)\times\mathcal{P}(Lang)$ has CFP iff for all $\phi, \psi \in Lang$, we have $\sim(\phi \wedge \psi) \in \mathsf{L}$ iff at least one of $\sim\phi, \sim\psi$ is in $\mathsf{L}$.

All of the logics that we will consider in this paper will have a common property, namely, that their sets of consequences can be seen as induced by  (1) their intended Kripke semantics, and, on the other hand by (2) their complete Hilbert-style axiomatizations. The first of these circumstances allows us to alternatively conceptualize logics as the values returned by an operator $\mathsf{L}$ applied to a tuple of the form $(Lang, Mod, \triangleright)$, where $Lang$ is the language of $\mathsf{L}$, $Mod$ is a class of Kripke models such that every $\mu \in Mod$ has an underlying set denoted by $\lvert\mu\rvert$, and $\triangleright \subseteq PMod\times Lang$ (where we set $PMod := \{(\mu, w)\mid \mu \in Mod,\,w\in \lvert\mu\rvert\}$) is a satisfaction relation between the \textit{pointed models} based on $Mod$ and the formulas of the language. We will understand the logic $\mathsf{L}(Lang, Mod, \triangleright)$ resulting from this application, as the subset of $\mathcal{P}(Lang)\times\mathcal{P}(Lang)$ such that $(\Gamma, \Delta) \in \mathsf{L}(Lang, Mod, \triangleright)$ iff
for no $(\mu, w)\in PMod$ do we have both $(\forall\phi\in \Gamma)(\mu, w\triangleright\phi)$ and $(\forall\psi\in \Delta)(\mu, w\not\triangleright\phi)$, which we will also express by saying that for no $(\mu, w)\in PMod$ do we have $\mu, w\rhd (\Gamma, \Delta)$. As this never leads to a confusion, we will also write $\mu, w\models_\mathsf{L}\phi$ and $\mu, w\models_\mathsf{L} (\Gamma, \Delta)$ meaning $\mu, w\triangleright\phi$ and $\mu, w\rhd (\Gamma, \Delta)$, respectively.

As for the Hilbert-style systems, all of them will be given by a finite number of axiomatic schemas $\alpha_1,\ldots,\alpha_n$ augmented with a finite number of inference rules $\rho_1,\ldots,\rho_m$, so the most general format sufficient for the present paper is $\Sigma(\alpha_1,\ldots,\alpha_n; \rho_1,\ldots,\rho_m)$. Just as with languages, all of the Hilbert-style systems considered in the present paper, happen to extend a certain minimal system which we will denote by $\mathbb{IL}^+$. We have $\mathbb{IL}^+:= \Sigma(\alpha_1-\alpha_8;\eqref{E:mp})$, where:
\begin{align*}
	&\phi \to(\psi\to\phi)\,(\alpha_1),\quad(\phi\to(\psi\to\chi))\to((\phi\to\psi)\to(\phi\to\chi))\,(\alpha_2),\\
	&\quad(\phi\wedge\psi)\to\phi\,(\alpha_3),\quad(\phi\wedge\psi)\to\psi\,(\alpha_4),\quad \phi\to(\psi\to (\phi\wedge\psi))\,(\alpha_5),\\
	&\quad\phi\to(\phi\vee \psi)\,(\alpha_6),\quad\psi\to(\phi\vee \psi)\,(\alpha_7),\quad (\phi\to\chi)\to((\psi \to \chi)\to ((\phi\vee\psi)\to \chi))\,(\alpha_8)
\end{align*}
and:
\begin{align}
	\text{From }\phi, \phi \to \psi&\text{ infer }\psi\label{E:mp}\tag{MP}
\end{align}
It is therefore important for our purposes to be able to refer to Hilbert-style systems as extensions of other systems. If $\mathbb{S} = \Sigma(\alpha_1,\ldots,\alpha_n; \rho_1,\ldots,\rho_m)$, and $\beta_1,\ldots,\beta_k$ are some new axiomatic schemes and $\sigma_1,\ldots,\sigma_r$ are some new rules, then we will write $\mathbb{S}+(\beta_1,\ldots,\beta_k;\sigma_1,\ldots,\sigma_r)$ to denote the system 
$$
\Sigma(\alpha_1,\ldots,\alpha_n,\beta_1,\ldots,\beta_k;\rho_1,\ldots,\rho_m,\sigma_1,\ldots,\sigma_r).
$$
Axiomatic systems can be viewed as operators generating logics when applied to languages. More precisely, if $\mathbb{S} = \Sigma(\alpha_1,\ldots,\alpha_n; \rho_1,\ldots,\rho_m)$ and $Lang$ is a language then $\mathsf{L} = \mathbb{S}[Lang]$ can be described as follows. We say that a $\phi \in Lang$ is \textit{provable} in $\mathbb{S}$ iff there exists a finite sequence $\psi_1,\ldots,\psi_k$ of formulas in $Lang$ such that every formula in this sequence is either a substitution instance of one of $\alpha_1,\ldots,\alpha_n$ or results from an application of one of $\rho_1,\ldots,\rho_m$ to some earlier formulas in the sequence and $\psi_k = \phi$; we will say that $(\Gamma, \Delta) \in \mathbb{S}[Lang]$ iff $(\Gamma, \Delta) \in \mathcal{P}(Lang)\times\mathcal{P}(Lang)$ and there exists a sequence $\chi_1,\ldots,\chi_r$ of formulas in $Lang$ such that every formula in it is either in $\Gamma$, or is provable in $\mathbb{S}$ or results from an application of \eqref{E:mp} to a pair of earlier formulas in the sequence, and, for some $\theta_1,\ldots,\theta_s\in \Delta$ we have $\chi_r = \theta_1\vee\ldots\vee\theta_s$. This definition makes sense in the context of our paper, since every language that we are going to consider contains $\vee$, and every axiomatic system that we are going to consider contains \eqref{E:mp}. We will also express the fact that  $(\Gamma, \Delta) \in \mathbb{S}[Lang]$ by writing $\Gamma\vdash_\mathbb{S}\Delta$.

Note that it follows from this definition that $s > 0$; therefore, $(\Gamma, \Delta) \in \mathbb{S}[Lang]$ implies that $\Delta \neq \emptyset$. This circumstance reveals a subtle gap between a logic $\mathsf{L}$ (over some language $Lang$) that is given by its Kripke semantics and a Hilbert-style system $\mathbb{S}$ that purports to capture $\mathsf{L}$. This gap, however, does not arise for $\mathsf{N4}$, our basic propositional logic, since, as we will see shortly, $\mathsf{N4}$ contains no pair of the form $(\Gamma, \emptyset)$. As for the other logics, we will use the method that works well for both intuitionistic and classical logic: $\mathbb{S}^\iota[Lang]:= \mathbb{S}[Lang]\cup \{(\Gamma, \emptyset)\mid (\Gamma, \{p_1\wedge\sim p_1\})\in \mathbb{S}[Lang])\}$.
  
Finally, we will also speak of the rules \textit{derivable} in a logic generated by an axiomatic system. If $\Gamma\cup \{\phi\}\Subset Lang$, then we will say that $\phi$ is derivable from $\Gamma$ in $\mathbb{S}[Lang]$ (and will write $\Gamma\vDdash_\mathbb{S}\phi$) iff there exists a finite sequence $\psi_1,\ldots,\psi_k$ of formulas in $Lang$ such that every formula in this sequence is either in $\Gamma$, or is provable in $\mathbb{S}$, or results from an application of one of $\rho_1,\ldots,\rho_m$ to some earlier formulas in the sequence, and $\psi_k = \phi$. It is easy to see that, for a $\mathsf{L} = \mathbb{S}[Lang]$ and a $\phi \in Lang$, we will have $\phi \in \mathsf{L}$ iff $\vdash_\mathbb{S}\phi$ iff $\vDdash_\mathbb{S}\phi$.

\section{Propositional basis and its logics}\label{S:basis}
Before we introduce the basic paraconsistent Nelsonian logic of conditionals, we recall some facts about the logic $\mathsf{N4}$, the paraconsistent version of Nelson's propositional logic of strong negation. We start by describing its Kripke semantics:
\begin{definition}\label{D:n4-model}
	A \textit{Nelsonian model} is a structure of the form $\mathcal{M} = (W, \leq, V^+, V^-)$, where $W \neq \emptyset$ is a set of worlds, $\leq$ is a pre-order (i.e., a reflexive and transitive relation) on $W$, and, for a $\star \in \{+, -\}$,  $V^\star:Var\to \mathcal{P}(W)$ is such that
	\begin{equation}\label{Cond:mon}\tag{mon}
	(\forall p \in Var)(\forall w,v \in W)((w\mathrel{\leq}v\,\&\,w \in V^\star(p))\Rightarrow  v \in V^\star(p))	
	\end{equation} 
A Nelsonian model is \textit{extended} iff $V^\star:Var\cup Var'\to \mathcal{P}(W)$ for every $\star \in \{+, -\}$.
\end{definition}
We will denote the class of all Nelsonian models by $Nel$. An (extended) intuitionistic model is any structure obtained from an (extended) Nelsonian model $\mathcal{M} = (W, \leq, V^+, V^-)$ by omitting $V^-$, its last element. A classical model is an intuitionistic model $\mathcal{M} = (W, \leq, V)$ where $W$ is a singleton. We denote the classes of (extended) intuituionistic (resp. classical) models by $Int$ (resp. $Int^e$, $Cl$).

For the logic $\mathsf{N4}$, we consider two satisfaction relations instead of one; we denote these relations by $\models^+_{\mathsf{N4}}$ and $\models^-_{\mathsf{N4}}$, respectively, and interpret them as representing verification and falsifcation of formulas at a given pointed Nelsonian model. They are defined by the following induction on the construction of a formula in $\mathcal{L}$:
\begin{align}
	\mathcal{M}, w&\models_{\mathsf{N4}}^\star p \text{ iff } w \in V^\star(p)\qquad\qquad\qquad\qquad\text{for $p$ atomic, $\star \in \{ +, -\}$}\label{Cl:atom}\tag{atom}\\
	\mathcal{M}, w&\models_{\mathsf{N4}}^+ \psi \wedge \chi \text{ iff } \mathcal{M}, w\models_{\mathsf{N4}}^+ \psi\text{ and }\mathcal{M}, w\models_{\mathsf{N4}}^+ \chi\label{Cl:con+}\tag{$\wedge+$}\\
	\mathcal{M}, w&\models_{\mathsf{N4}}^- \psi \wedge \chi \text{ iff } \mathcal{M}, w\models_{\mathsf{N4}}^- \psi\text{ or }\mathcal{M}, w\models_{\mathsf{N4}}^- \chi\label{Cl:con-}\tag{$\wedge-$}\\
	\mathcal{M}, w&\models_{\mathsf{N4}}^+ \psi \vee \chi \text{ iff } \mathcal{M}, w\models_{\mathsf{N4}}^+ \psi\text{ or }\mathcal{M}, w\models_{\mathsf{N4}}^+ \chi\label{Cl:dis+}\tag{$\vee+$}\\
	\mathcal{M}, w&\models_{\mathsf{N4}}^- \psi \vee \chi\text{ iff } \mathcal{M}, w\models_{\mathsf{N4}}^- \psi\text{ and }\mathcal{M}, w\models_{\mathsf{N4}}^- \chi\label{Cl:dis-}\tag{$\vee-$}\\
	\mathcal{M}, w&\models_{\mathsf{N4}}^+ \sim\psi \text{ iff } \mathcal{M}, w\models_{\mathsf{N4}}^- \psi\label{Cl:neg+}\tag{$\sim+$}\\
	\mathcal{M}, w&\models_{\mathsf{N4}}^- \sim\psi\text{ iff } \mathcal{M}, w\models_{\mathsf{N4}}^+ \psi\label{Cl:neg-}\tag{$\sim-$}\\
	\mathcal{M}, w&\models_{\mathsf{N4}}^+ \psi \to \chi \text{ iff } (\forall v \geq w)(\mathcal{M}, v\models_{\mathsf{N4}}^+ \psi\text{ implies }\mathcal{M}, v\models_{\mathsf{N4}}^+ \chi)\label{Cl:im+}\tag{$\to+$}\\
	\mathcal{M}, w&\models_{\mathsf{N4}}^- \psi \to \chi \text{ iff } \mathcal{M}, w\models_{\mathsf{N4}}^+ \psi\text{ and }\mathcal{M}, w\models_{\mathsf{N4}}^- \chi\label{Cl:im-}\tag{$\to-$}	
\end{align}
Even though $\mathsf{N4}$ has two satisfaction relations, the presence of the strong negation $\sim$ allows to reflect, as it were, the falsifications given by $\models^-_{\mathsf{N4}}$ within the structure of $\models^+_{\mathsf{N4}}$ identifying them with the verifications of negated formulas. Thus the usual definition of $\mathsf{N4}$ in terms of its Kripke semantics prioritizes $\models^+_{\mathsf{N4}}$ and, presented in the terminology introduced in the previous section, looks as follows:
$$
\mathsf{N4}:= \mathsf{L}(\mathcal{L}, Nel, \models^+_{\mathsf{N4}})
$$
Note that every subset of $\mathcal{L}$ is satisfiable in $\mathsf{N4}$; in other words:
\begin{proposition}\label{P:n4-satisfiable}
For every $\Gamma \subseteq \mathcal{L}$, $(\Gamma, \emptyset)\notin \mathsf{N4}$.
\end{proposition} 
\begin{proof}[Proof (a sketch)]
	Consider $\mathcal{M}:= (\{w\}, \{(w,w)\}, \{(p, \{w\}\mid p \in Var)\}, \{(p, \{w\}\mid p \in Var)\})$. An easy induction on the construction of the formula shows that we have $\mathcal{M}, w\models_{\mathsf{N4}}^+ \phi$ for every $\phi \in \mathcal{L}$. Therefore $(\Gamma, \emptyset) \notin \mathsf{N4}$ for every $\Gamma \subseteq \mathcal{L}$.
\end{proof}
$\mathsf{N4}$ is one of relatively well-researched non-classical propositional logics (see, e.g., \cite[Ch. 8]{odintsov} and \cite[Ch. 9.7a ff]{priest}; a very accessible exposition can be also found in \cite{pearce}).

A peculiar feature of Nelson's logic is that the contraposition fails. However, it is possible to strengthen the verification condition for implication so as to allow for contraposition, while leaving the falsification condition untouched; one can then also define a version of equivalence based on this stronger implication. More precisely, consider the following abbreviations (meant to apply to any language $\mathcal{L}'\supseteq \mathcal{L}$):
\begin{itemize}
	\item $\phi \leftrightarrow \psi$ (equivalence) for $(\phi \to \psi)\wedge(\psi\to \phi)$.
	
	\item $\phi \Rightarrow \psi$ (strong implication) for $(\phi \to \psi)\wedge(\sim\psi\to \sim\phi)$.
	
	\item $\phi \Leftrightarrow \psi$ (strong equivalence) for $(\phi \Rightarrow \psi)\wedge(\psi\Rightarrow \phi)$.
\end{itemize} 
It is easy to see that the verification and falsification conditions for $\Rightarrow$ look as follows:
\begin{align*}
	\mathcal{M}, w&\models_{\mathsf{N4}}^+ \psi \Rightarrow \chi \text{ iff } (\forall v \geq w)((\mathcal{M}, v\models_{\mathsf{N4}}^+ \psi\text{ implies }\mathcal{M}, v\models_{\mathsf{N4}}^+ \chi)\text{ and }(\mathcal{M}, v\models_{\mathsf{N4}}^- \chi\text{ implies }\mathcal{M}, v\models_{\mathsf{N4}}^- \psi))\\
	\mathcal{M}, w&\models_{\mathsf{N4}}^- \psi \Rightarrow \chi \text{ iff } \mathcal{M}, w\models_{\mathsf{N4}}^+ \psi\text{ and }\mathcal{M}, w\models_{\mathsf{N4}}^- \chi
\end{align*}
The lack of contraposition for $\to$ is also the reason why the equivalents are in general not substitutable for one another in Nelson's logic. However, strong equivalents are not only mutually substitutable, but also ensure the strong equivalence of the resulting formulas. Thus not only the strong implication, but also the strong equivalence $\Leftrightarrow$ plays an important role in $\mathsf{N4}$. We summarize these facts in the following proposition:
\begin{proposition}\label{P:n4-basics}
Let $p, q \in Var$ and let $\phi, \psi\in \mathcal{L}$. Then the following statements hold:
\begin{enumerate}
	\item $(\phi \Rightarrow \psi)\Rightarrow(\sim\psi\Rightarrow\sim\phi)\in \mathsf{N4}$.
	
	\item However, the same cannot be said about $\to$, since we have $(p \rightarrow q)\rightarrow(\sim q\rightarrow\sim p)\notin \mathsf{N4}$.
	
	\item We have $(\phi \Rightarrow \psi)\to(\phi\rightarrow\psi)\in \mathsf{N4}$, but not vice versa, so that $\Rightarrow$ is stronger than $\to$.
	
	\item If $(\phi\Leftrightarrow\psi) \in \mathsf{N4}$, and $\theta, \theta_\phi, \theta_\psi \in \mathcal{L}$ are such that $\theta_\phi$ and  $\theta_\psi$ are obtained from $\theta$ by replacing a $p \in Var$ with $\phi$ and $\psi$, respectively, then also $(\theta_\phi\Leftrightarrow\theta_\psi) \in \mathsf{N4}$.
	
	\item However, the same cannot be said about $\leftrightarrow$ since we have $\sim(p \to q)\leftrightarrow(p \wedge \sim q) \in \mathsf{N4}$, but $\sim\sim(p \to q)\leftrightarrow \sim(p \wedge \sim q) \notin \mathsf{N4}$.
\end{enumerate}
\end{proposition}
We omit the obvious proof.

In what follows, we will also end up mentioning several other logics based on similar sets of connectives. First of them is intuitionistic logic which we define as $\mathsf{IL}:= \mathsf{L}(\mathcal{L}, Int, \models_{\mathsf{IL}})$, where $\models_{\mathsf{IL}}$ is given by clauses \eqref{Cl:atom}, \eqref{Cl:con+}, \eqref{Cl:dis+}, \eqref{Cl:im+} plus the following alternative clause for negation:
\begin{align}
	\mathcal{M}, w&\models_{\mathsf{IL}} \sim\psi \text{ iff } (\forall v \geq w)(\mathcal{M}, v\not\models_{\mathsf{IL}} \psi)\label{Cl:negi}\tag{$\sim i$}
\end{align}
The second one is the \textit{extended positive intuitionistic logic} given by $\mathsf{IL}^{e+}:= \mathsf{L}(\mathcal{L}^{e+}, Int^e, \models_{\mathsf{IL}})$. The third system is the \textit{classical logic}, which can be defined as $\mathsf{CL}:= \mathsf{L}(\mathcal{L}, Cl, \models_{\mathsf{IL}})$. 

All of the logics introduced thus far can be alternatively defined by means of Hilbert-style systems. More precisely, consider the following axiomatic schemes:
\begin{align}
	\sim\sim\phi &\leftrightarrow \phi\label{E:a0.1}\tag{An1}\\
	\sim(\phi\wedge \psi) 
	&\leftrightarrow (\sim\phi\vee \sim\psi)\label{E:a0.2}\tag{An2}\\
	\sim(\phi\vee \psi) &\leftrightarrow (\sim\phi\wedge \sim\psi)\label{E:a0.3}\tag{An3}\\
	\sim(\phi\to \psi) &\leftrightarrow (\phi\wedge \sim\psi)\label{E:a0.4}\tag{An4}\\
(\phi\to \psi)&\to((\phi\to\sim\psi)\to \sim\phi)\label{E:a0.5}\tag{An5}\\
		\phi&\to(\sim\phi\to\psi)\label{E:a0.6}\tag{An6}\\
	\phi&\vee\sim\phi\label{E:a0.7}\tag{An7}
\end{align} 
We can now introduce the following axiomatic systems:
\begin{align*}
	\mathbb{N}4:= (\mathbb{IL}^+)+(\eqref{E:a0.1}-\eqref{E:a0.4};),\,\mathbb{IL}:= (\mathbb{IL}^+)+(\eqref{E:a0.5},\eqref{E:a0.6};),\,\mathbb{CL}:= (\mathbb{IL})+(\eqref{E:a0.7};).
\end{align*}
The following facts are well-known:
\begin{proposition}\label{P:prop}
	The following equations hold:
	\begin{enumerate}
		\item $\mathsf{N4} = \mathbb{N}4[\mathcal{L}]$.
		
		\item $\mathsf{IL}^{e+} = (\mathbb{IL}^+)[\mathcal{L}^{e+}]$.
		
		\item $\mathsf{IL} = \mathbb{IL}^\iota[\mathcal{L}]$.
		
		\item $\mathsf{CL} = \mathbb{CL}^\iota[\mathcal{L}] = (\mathbb{N}4+(\eqref{E:a0.6},\eqref{E:a0.7};))^\iota[\mathcal{L}] = ((\mathbb{IL}^+)+(\eqref{E:a0.6},\eqref{E:a0.7};))^\iota[\mathcal{L}]$.	
	\end{enumerate}
\end{proposition}
Furthermore, consider the mapping $E:\mathcal{L}\to\mathcal{L}^{e+}$ defined by the following induction on the construction of $\phi \in \mathcal{L}$:
\begin{align*}
	E(p_i)&:= p_i&&i \in \omega\\
	E(\sim p_i)&:= q_i&&i \in \omega\\
	E(\phi\wedge \psi)&:= E(\phi)\wedge E(\psi)\\
	E(\sim(\phi \wedge \psi))&:= E(\sim\phi)\vee E(\sim\psi)\\
	E(\phi\vee \psi)&:= E(\phi)\vee E(\psi)\\
	E(\sim(\phi \vee \psi))&:= E(\sim\phi)\wedge E(\sim\psi)\\
	E(\sim\sim\phi)&:= E(\phi)\\
	E(\phi \to \psi)&:= E(\phi)\to E(\psi)\\
	E(\sim(\phi \to \psi))&:= E(\phi)\wedge E(\sim\psi)
\end{align*}
It is known that $E$ faithfully embeds $\mathsf{N4}$ into $\mathsf{IL}^{e+}$; in other words one can prove:
\begin{proposition}\label{P:prop-embedding}
	For all $\Gamma, \Delta \subseteq \mathcal{L}$, we have $\Gamma\models_{\mathsf{N4}}\Delta$ iff $E(\Gamma)\models_{\mathsf{IL}^{e+}}E(\Delta)$.
\end{proposition}
\begin{proof}[Proof (a sketch)]
We prove the following claims:
	
	\textit{Claim 1}. Let $\mathcal{M} = (W, \leq, V^+, V^-)\in Nel$. Then $\mathcal{M}^i = (W, \leq, V)\in Int^e$  where, for an $i\in \omega$,  we set $V(p_i): = V^+(p_i)$ and $V(q_i) := V^-(p_i)$. Moreover, we have:
	$$
	\mathcal{M}, w \models_{\mathsf{N4}}^+ \phi\text{ iff }\mathcal{M}^i, w \models_{\mathsf{IL}} E(\phi)
	$$
	for every $w \in W$ and every $\phi \in \mathcal{L}$.
	
	The proof of Claim 1 is straightforward; the proof of its main bi-conditional proceeds by induction on the construction of $\phi$ in which we only consider a couple of cases.
	
	Assume $\phi = \psi \to \chi$. Then $\mathcal{M}, w \models_{\mathsf{N4}}^+ \phi$ iff $(\forall v \geq v)(\mathcal{M}, v \models_{\mathsf{N4}}^+ \psi\text{ implies }\mathcal{M}, v \models_{\mathsf{N4}}^+ \chi)$ iff, by IH, $(\forall v \geq v)(\mathcal{M}^i, v \models_{\mathsf{IL}} E(\psi)\text{ implies }\mathcal{M}^i, v \models_{\mathsf{IL}} E(\chi))$ iff $\mathcal{M}^i, w \models_{\mathsf{IL}} E(\phi)$.
	
	On the other hand, $\mathcal{M}, w \models_{\mathsf{N4}}^+ \sim\phi$ iff $\mathcal{M}, w \models_{\mathsf{N4}}^+ \psi\text{ and }\mathcal{M}, v \models_{\mathsf{N4}}^- \chi$ iff $\mathcal{M}, w \models^+ \psi\text{ and }\mathcal{M}, v \models_{\mathsf{N4}}^+ \sim\chi$, iff, by IH, $\mathcal{M}^i, v \models_{\mathsf{IL}} E(\psi)\text{ and }\mathcal{M}^i, v \models_{\mathsf{IL}} E(\sim\chi)$ iff $\mathcal{M}^i, w \models_{\mathsf{IL}} E(\sim\phi)$.  
	
	\textit{Claim 2}. Let $\mathcal{M} = (W, \leq, V)\in Int^e$. Then $\mathcal{M}^{n4} = (W, \leq, V^+, V^-)\in Nel$,  where, for an $i \in \omega$, we set $V^+(p_i): = V(p_i)$ and $V^-(p_i) := V(q_i)$ is an $\mathsf{N4}$-model, and, moreover, we have:
	$$
	\mathcal{M}^{n4}, w \models_{\mathsf{N4}}^+ \phi\text{ iff }\mathcal{M}, w \models_{\mathsf{IL}} E(\phi)
	$$
	for every $w \in W$ and every $\phi \in \mathcal{L}$.
	
	The proof is similar to the proof of Claim 1.
	
	Now, if $\mathcal{M}, w \models_{\mathsf{N4}} (\Gamma, \Delta)$, then, by Claim 1, $\mathcal{M}^i, w\models_{\mathsf{IL}}(E(\Gamma), E(\Delta))$. Conversely, if $\mathcal{M}, w\models_{\mathsf{IL}}(E(\Gamma),E(\Delta))$, then, by Claim 2, $\mathcal{M}^{n4}, w\models_{\mathsf{N4}}(\Gamma, \Delta)$.   
\end{proof}
We end this section with a summary of the known facts about DP and CFP for the introduced logics:
\begin{proposition}\label{P:dp-basics}
	The following statements are true:
	\begin{enumerate}
		\item DP is satisfied by $\mathsf{N4}$, $\mathsf{IL}$ and $\mathsf{IL}^+$, and failed by $\mathsf{CL}$.
		
		\item CFP is satisfied by  $\mathsf{N4}$ and failed by $\mathsf{IL}$ and $\mathsf{CL}$.		
	\end{enumerate}
\end{proposition}
 
\section{$\mathsf{N4CK}$, a basic paraconsistent Nelsonian logic of conditionals}\label{S:system}
\subsection{Language and semantics}\label{sub:lands}
We are going to define $\mathsf{N4CK}$, a basic paraconsistent Nelsonian logic of conditionals over the language $\mathcal{L}_{\boxto}$; the might-conditional $\diamondto$ will not be considered an elementary connective in $\mathsf{N4CK}$ but will be introduced as an abbreviation instead:
\begin{itemize}
	\item $\phi\diamondto\psi$ (might-conditional) is an abbreviation for $\sim(\phi\boxto\sim\psi)$.
\end{itemize}
The formulas of $\mathcal{L}_{\boxto}$ are interpreted by the following models\footnote{The models defined are the so-called \textit{Chellas models}, introduced in \cite{chellas}. An eqiuvalent (at least as far as expressivity of $\mathcal{L}_{\boxto}$ goes) semantics could have been given by \textit{Segerberg models}, introduced in \cite{segerberg}, which use a designated family of pairs of subsets of $W$ in place of $\mathcal{P}(W)\times\mathcal{P}(W)$ as in Definition \ref{D:model}. Yet another alternative is to use families of formula-indexed binary relations $\{R_\phi\mid\phi\in\mathcal{L}\}$. We have chosen Chellas models since their definition looks short and simple; yet the Segerberg models can be perhaps ascribed a deeper foundational meaning.}:

\begin{definition}\label{D:model}
	A Nelsonian conditional model is a structure of the form $\mathcal{M} = (W, \leq, R, V^+, V^-)$, $(W, \leq, V^+, V^-)\in Nel$ and $R \subseteq W \times (\mathcal{P}(W)\times \mathcal{P}(W))\times W$. Thus, for all $X,Y \subseteq W$, $R$ induces a binary relation $R_{(X,Y)}$ on $W$ such that, for all $w,v \in W$, $R_{(X,Y)}(w,v)$ iff $R(w, (X,Y), v)$. Finally, the following conditions must be satisfied for all $X, Y \subseteq W$:
	\begin{align}
		(\leq^{-1}\circ R_{(X,Y)}) &\subseteq (R_{(X,Y)}\circ\leq^{-1})\label{Cond:1}\tag{c1}\\
		(R_{(X,Y)}\circ\leq) &\subseteq (\leq\circ R_{(X,Y)})\label{Cond:2}\tag{c2}
	\end{align}	
\end{definition}
Conditions \eqref{Cond:1} and \eqref{Cond:2} can be reformulated as requirements to complete the dotted parts of each of the following diagrams once the respective straight-line part is given:
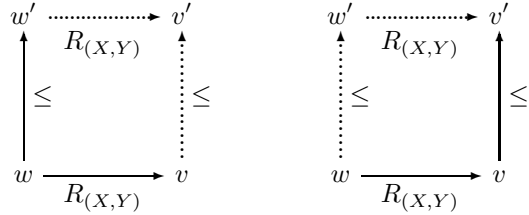
\begin{center}
\begin{diagram}
	w'& \rDotsto_{R_{(X,Y)}} & v'         & & w'& \rDotsto_{R_{(X,Y)}} & v'\\
	\uTo_\leq &     & \uDotsto_\leq & & \uDotsto_\leq & &  \uTo_\leq\\
	w & \rTo_{R_{(X,Y)}} & v              & & w        &  \rTo_{R_{(X,Y)}} & v
\end{diagram}	
\end{center}
Our standard notation for models is $\mathcal{M} = (W, \leq, R, V^+, V^-)$. Any model decorations are assumed to be inherited by their components, so that, for example $\mathcal{M}_n$ always stands for  $(W_n, \leq_n, R_n, V^+_n, V^-_n)$. The class of all conditional Nelsonian models will be denoted by $CNel$. Just as in the case of $\mathsf{N4}$, we define two satisfaction relations for $\mathsf{N4CK}$, denoted by $\models^+$ and $\models^-$. Again, we associate them with the notions of \textit{verification} and \textit{falsification}, respectively, and define them by induction on the construction of $\phi \in \mathcal{L}_{\boxto}$ by means of all the clauses that we have given earlier for $\models_\mathsf{N4}^+$ and $\models_\mathsf{N4}^-$, respectively, plus the following clauses for $\boxto$:
\begin{align*}
	\mathcal{M}, w&\models^+ \psi \boxto \chi \text{ iff } (\forall v \geq w)(\forall u \in W)(R_{\|\psi\|_\mathcal{M}}(v, u) \text{ implies }\mathcal{M}, u\models^+ \chi)\\
	\mathcal{M}, w&\models^- \psi \boxto \chi \text{ iff } (\exists u \in W)(R_{\|\psi\|_\mathcal{M}}(w, u)\text{ and }\mathcal{M}, u\models^- \chi)	
\end{align*}
where we assume, for any given $\phi \in \mathcal{L}_{\boxto}$, that:
$$
\|\phi\|_\mathcal{M}: = (\|\phi\|^+_\mathcal{M}, \|\phi\|^-_\mathcal{M}) = (\{w \in W\mid \mathcal{M}, w\models^+\phi\}, \{w \in W\mid \mathcal{M}, w\models^-\phi\}).
$$
It is instructive to compute the inductive clauses for $\diamondto$ as well:
\begin{align*}
	\mathcal{M}, w&\models^+ \psi \diamondto \chi \Leftrightarrow (\exists u \in W)(R_{\|\psi\|_\mathcal{M}}(w, u)\text{ and }\mathcal{M}, u\models^+ \chi)\\
	\mathcal{M}, w&\models^- \psi \diamondto \chi \text{ iff } (\forall v \geq w)(\forall u \in W)(R_{\|\psi\|_\mathcal{M}}(v, u) \text{ implies }\mathcal{M}, u\models^- \chi)	
\end{align*}
We now define $\mathsf{N4CK}:= \mathsf{L}(\mathcal{L}_{\boxto}, CNel, \models^+)$. Throughout this section, we will write $\Gamma\models\Delta$ and $\mathcal{M}, w \models (\Gamma, \Delta)$, meaning $\Gamma\models_{\mathsf{N4CK}}\Delta$ and $\mathcal{M}, w \models_{\mathsf{N4CK}} (\Gamma, \Delta)$, respectively.
\begin{remark}\label{R:1}
	In conditional logic, it is customary to generate the binary accessibility relations using truth-sets of the formulas. In the case of $\mathsf{N4CK}$, following this idea would entail defining $R$ to be a subset of $W \times \mathcal{P}(W)\times W$ instead of $W \times (\mathcal{P}(W)\times \mathcal{P}(W))\times W$ and setting $\|\phi\|_\mathcal{M}:= \{w \in W\mid \mathcal{M}, w\models^+\phi\}$; let us denote the resulting logic by $\mathsf{N4CK}'$ Note that $\mathsf{N4CK}'$ can hardly be called a \textit{minimal} conditional expansion of $\mathsf{N4}$, since we have $(\phi\boxto\chi)\Leftrightarrow(\phi\boxto\chi) \in \mathsf{N4CK}'$ whenever $\phi\leftrightarrow\psi\in \mathsf{N4CK}'$. By Proposition \ref{P:n4-basics}.5, we know that in $\mathsf{N4}$ the substitution of simple equivalents generally fails to guarantee even the simple equivalence of the resulting substitution instances, let alone their strong equivalence. Imposing this very strong substitution rule for the antecedents of $\boxto$ would certainly stand in need of a separate argument.
	
	On the other hand, we will see in the next subsection that, in $\mathsf{N4CK}$, the result of a substitution in a conditional context is provably strongly equivalent to the original formula only in case the substitution itself is a substitution of strong equivalents and that, therefore, the substitution format for $\boxto$ in $\mathsf{N4CK}$ is set in accordance with the pattern given in Proposition \ref{P:n4-basics}.4.\footnote{To the best of our knowledge, the idea to use the bi-set-indexed accessibility relations in the semantics of $\mathsf{FDE}$-based conditional logics was first expressed in \cite[Remark 1]{wu}; later a similar idea was used by the authors of \cite{ds} in setting up the neighborhood semantics for some $\mathsf{FDE}$-based modal logics. The logic $\mathsf{FDE}$ of \textit{first-degree entailment} is a well-studied non-classical logic introduced in \cite{anderson}; it can be obtained by omitting $\to$ from $\mathsf{N4}$.}
\end{remark}
Given an $\mathcal{M}\in CNel$ and an $X \subseteq W$, we say that $X$ is \textit{upward-closed in $\mathcal{M}$} iff $(\forall w \in X)(\forall v \geq w)(v \in X)$. Definition \ref{D:model} clearly implies that, for every $\mathcal{M}\in CNel$ and for every $p \in Var$, both elements of $\|p\|_\mathcal{M} = (V^+(p), V^-(p))$ are upward-closed in $\mathcal{M}$. The latter observation can be lifted to the level of arbitrary formulas:
\begin{lemma}\label{L:chellas-monotonicity}
Let $\mathcal{M}\in CNel$, and let $\phi \in \mathcal{L}_{\boxto}$. Then both elements of $\|\phi\|_\mathcal{M}$ are upward-closed in $\mathcal{M}$. 
\end{lemma}
We omit the easy proof by induction on the construction of $\phi \in \mathcal{L}_{\boxto}$. It is easy to see, next, that Proposition \ref{P:n4-satisfiable} can be extended to $\mathsf{N4CK}$:
\begin{proposition}\label{P:satisfiable}
For every $\Gamma_{\boxto} \subseteq \mathcal{L}$, $(\Gamma, \emptyset)\notin \mathsf{N4CK}$.	
\end{proposition}
\begin{proof}[Proof (a sketch)]
	We consider model $\mathcal{M}$ as defined in the proof of Proposition \ref{P:n4-satisfiable} and extend it with the relation $R:= \{(w, (\{w\}, \{w\}), w)\}$. The rest of the argument is as in the proof of Proposition \ref{P:n4-satisfiable}.
\end{proof}
We close this subsection by extending the statements made in Proposition \ref{P:dp-basics} about $\mathsf{N4}$ to $\mathsf{N4CK}$:
\begin{proposition}\label{P:dp}
$\mathsf{N4CK}$ satisfies both DP and CFP.	
\end{proposition}
\begin{proof}
	We argue for the satisfaction of DP first. The right-to-left direction is trivial. As for the other direction, assume, towards contradiction, that $\phi_1\vee\phi_2 \in \mathsf{N4CK}$, but both  $\phi_1 \notin \mathsf{N4CK}$ and $\phi_2 \notin \mathsf{N4CK}$. Then we can choose pointed models $(\mathcal{M}_1,w_1)$ and $(\mathcal{M}_2,w_2)$ such that $\mathcal{M}_i,w_i\not\models_c\phi_i$ for all $i \in \{1,2\}$; we may assume, wlog, that $W_1\cap W_2 = \emptyset$. 
	We then choose an element $w$ outside $W_1\cup W_2$ and define the following pointed model $(\mathcal{M},w)$ for which we set:
	\begin{align*}
		W &:= \{w\}\cup W_1\cup W_2\\
		\leq &:= \{(w,v)\mid v \in W\} \cup \leq_1 \cup \leq_2\\
		R &:= \{(v,(X, Y), u)\mid v,u \in W_1,\,(v,(X\cap W_1, Y\cap W_1), u)\in R_1\}\cup\\
		&\qquad\qquad\qquad\qquad\cup \{(v,(X,Y),u)\mid v,u \in W_2,\,(v,(X\cap W_2, Y\cap W_2), u)\in R_2\}\\
		V^\star(p) &:= V^\star_1(p)\cup V^\star_2(p)\qquad\qquad\qquad \text{for }p \in Var,\,\star\in\{+,-\}
	\end{align*}
We show that $\mathcal{M}\in CNel$. The only non-trivial part is the satisfaction of conditions \eqref{Cond:1} and \eqref{Cond:2} from Definition \ref{D:model}.

As for \eqref{Cond:1}, assume that some $v',v,u \in W$ and $X, Y\subseteq W$ are such that $v'\geq v\mathrel{R_{(X,Y)}}u$. Then, by defintion of $R$, we must have either $v,u \in W_1$ or $v,u\in W_2$. Assume, wlog, that $v,u \in W_1$. Then we must have, first, that $v\mathrel{(R_1)_{(X\cap W_1, Y\cap W_1)}}u$, and, second, that $v'\geq_1 v$ so that also $v'\in W_1$. But then, since $\mathcal{M}_1$ satisfies  \eqref{Cond:1}, there must be a $u'\in W_1$ such that $v'\mathrel{(R_1)_{(X\cap W_1, Y\cap W_1)}}u'\geq_1u$ whence clearly also  $v'\mathrel{R_{(X,Y)}}u'\geq u$, so that  \eqref{Cond:1} is shown to hold for $\mathcal{M}$. We argue similarly for  \eqref{Cond:2}.

The following claim can be shown by induction on the construction of $\phi \in\mathcal{L}_{\boxto}$:

\textit{Claim}. For every $\star \in \{+, -\}$, every $i\in \{1,2\}$, every $v \in W_i$, and every $\phi \in \mathcal{L}_{\boxto}$, we have $\mathcal{M}, v\models^\star\phi$ iff $\mathcal{M}_i, v\models^\star\phi$. 

It follows now that $\mathcal{M}, w_i\not\models\phi_i$ for all $i \in \{1,2\}$, and, since we have $w\leq w_1, w_2$, Lemma \ref{L:chellas-monotonicity} implies that  $\mathcal{M}, w\not\models\phi_i$ for all $i \in \{1,2\}$, or, equivalently, that $\mathcal{M}, w\not\models_c\phi_1\vee\phi_2$, contrary to our assumption. The obtained contradiction shows that $\mathsf{N4CK}$ must have DP.

The argument for CFP is similar.
\end{proof}

\subsection{Axiomatization}\label{sub:axiomatization}
In this subsection, we obtain a sound and (strongly) complete axiomatization of $\mathsf{N4CK}$. We consider the Hilbert-style axiomatic system $\mathbb{N}4\mathbb{CK}$, for which we set $\mathbb{N}4\mathbb{CK}:= \mathbb{N}4+(\eqref{E:a1}-\eqref{E:a4};\eqref{E:RAbox},\eqref{E:RCbox1},\eqref{E:RCbox2})$, where we assume that:
\begin{align}
	((\phi \boxto \psi)\wedge(\phi \boxto \chi))&\Leftrightarrow(\phi \boxto (\psi \wedge \chi))\label{E:a1}\tag{A1}\\
(\sim(\phi\boxto\psi)\wedge (\phi \boxto \chi))&\to\sim(\phi \boxto(\psi\vee\sim\chi))\label{E:a2}\tag{A2}\\
	((\phi \diamondto \psi)\to(\phi \boxto \chi))&\to(\phi \boxto (\psi \to \chi))\label{E:a3}\tag{A3}\\
	\phi\boxto (\psi&\to\psi)\label{E:a4}\tag{A4}\\
	\text{From }\phi\Leftrightarrow\psi &\text{ infer } (\phi\boxto\chi)\Leftrightarrow(\psi\boxto\chi)\label{E:RAbox}\tag{RA$\Box$}\\
	\text{From }\phi\leftrightarrow\psi &\text{ infer } (\chi\boxto\phi)\leftrightarrow(\chi\boxto\psi)\label{E:RCbox1}\tag{RC$\Box$1}\\
	\text{From }\sim\phi\leftrightarrow\sim\psi &\text{ infer } \sim(\chi\boxto\phi)\leftrightarrow\sim(\chi\boxto\psi)\label{E:RCbox2}\tag{RC$\Box$2}	
\end{align}
Within this subsection, in addition to assuming all the notions introduced in Section \ref{S:Prel}, we will also write $\vdash$ and $\vDdash$ to mean $\vdash_{\mathbb{N}4\mathbb{CK}}$ and $\vDdash_{\mathbb{N}4\mathbb{CK}}$, respectively, to avoid the clutter.

Before we go on to prove the soundness and completeness of $\mathbb{N}4\mathbb{CK}$ relative to $\mathsf{N4CK}$, we would like to quickly address the relations between $\mathbb{N}4\mathbb{CK}$ and $\mathsf{N4}$: 
\begin{lemma}\label{L:n4}
	The following statements hold:
	\begin{enumerate}
		\item If $\Gamma, \Delta \subseteq\mathcal{L}$ are such that $\Gamma\models_{\mathsf{N4}}\Delta$, and  $\Gamma', \Delta' \subseteq \mathcal{L}_{\boxto}$ are obtained from $\Gamma, \Delta$ by a simultaneous substitution of $\mathcal{L}_{\boxto}$-formulas for variables, then $\Gamma'\vdash\Delta'$. Moreover, Deduction Theorem holds for $\mathbb{N}4\mathbb{CK}$ in that for all $\Gamma \cup \{\phi,\psi\}\subseteq \mathcal{L}_{\boxto}$ we have $\Gamma \vdash \phi\to\psi$ iff $\Gamma, \phi\vdash \psi$.
		
		\item If $\phi \in \mathcal{L}$, then $\vdash \phi$ iff $\phi\in \mathsf{N4}$.
\end{enumerate}
\end{lemma}
\begin{proof}[Proof (a sketch)]
	Part 1 is trivial. As for Part 2, its ($\Leftarrow$)-part is also trivial, and its ($\Rightarrow$)-part follows from the observation that, given a proof of $\phi \in \mathcal{L}_n$ in $\mathbb{N}4\mathbb{CK}$ we can turn it into a proof in $\mathbb{N}4$ by replacing all its subformulas of the form $\chi\boxto\theta$ with  $p\to p$ for some fixed $p \in Var$ and apply Proposition \ref{P:prop}.1. 
\end{proof}
Turning now to the relations between $\mathbb{N}4\mathbb{CK}$ and $\mathsf{N4CK}$, we observe, first, that  $\mathbb{N}4\mathbb{CK}$ only allows us to deduce theorems of $\mathsf{N4CK}$:
\begin{lemma}\label{L:soundness}
	For every $\phi\in\mathcal{L}_{\boxto}$, if $\vdash\phi$, then $\phi\in\mathsf{N4CK}$.
\end{lemma}
The proof proceeds by the usual method, i.e. we show that all the axioms are valid and that the rules of $\mathbb{N}4\mathbb{CK}$ preserve the validity. We are now going to show the converse of Lemma \ref{L:soundness}, and we start our work by proving some theorems and derived rules in $\mathbb{N}4\mathbb{CK}$, which we collect in the following lemma: 
\begin{lemma}\label{L:theorems}
	Let $\phi, \psi, \chi \in \mathcal{L}_{\boxto}$. The following theorems and derived rules can be deduced in $\mathbb{N}4\mathbb{CK}$:
	\begin{align}
		\phi \Leftrightarrow \psi &\vDdash (\chi \boxto \phi) \Leftrightarrow (\chi \boxto \psi)\label{E:RCbox}\tag{RC$\Box$}\\
		\phi &\vDdash (\psi\boxto \phi)\label{E:Rnec}\tag{Nec}\\
		(\phi \to \psi) &\vDdash (\chi\boxto\phi)\to(\chi\boxto\psi)\label{E:Rmbox}\tag{RM$\Box$}\\
		(\sim\phi \to \sim\psi) &\vDdash \sim(\chi\boxto\phi)\to\sim(\chi\boxto\psi)\label{E:Rnbox}\tag{RM$\sim\Box$}\\
		(\phi \to \psi) &\vDdash (\chi\diamondto\phi)\to(\chi\diamondto\psi)\label{E:Rmdiam}\tag{RM$\Diamond$}\\
		(\phi\boxto(\psi \to \chi))&\to((\phi\boxto\chi)\to(\phi\boxto\chi))\label{E:T1}\tag{T1}\\
		(\phi\boxto(\sim\psi \to \sim\chi))&\to(\sim(\phi\boxto\psi)\to\sim(\phi\boxto\chi))\label{E:T2}\tag{T2}\\
		((\phi\diamondto(\psi \to \sim\chi))&\wedge(\phi\boxto\psi))\to\sim(\phi\boxto\chi))\label{E:T3}\tag{T3}\\
		(\phi\diamondto(\psi \vee \chi))&\to((\phi\diamondto\psi)\vee(\phi\diamondto\chi))\label{E:T4}\tag{T4}\\
		(\phi\diamondto (\psi\to \chi))&\to((\phi\boxto\psi)\to(\phi\diamondto\psi))\label{E:T5}\tag{T5}\\
		\sim(\phi\boxto\psi)&\leftrightarrow(\phi\diamondto\sim\psi)\label{E:T6}\tag{T6}
	\end{align}
\end{lemma} 
The sketch of its proof is relegated to Appendix \ref{A:1}.

A bi-set $(\Gamma, \Delta)\in \mathcal{P}(\mathcal{L}_{\boxto})\times\mathcal{P}(\mathcal{L}_{\boxto})$ is called \textit{consistent} iff  $\Gamma \not\vdash \Delta$ and \textit{complete} iff $\Gamma\cup\Delta = \mathcal{L}_{\boxto}$; it is called \textit{maximal} iff it is both complete and consistent. Note that, in view of Lemma \ref{L:n4}.1, this definition allows for the following equivalent form:
\begin{lemma}\label{L:alt-consistency}
	A bi-set $(\Gamma, \Delta)\in \mathcal{P}(\mathcal{L}_{\boxto})\times\mathcal{P}(\mathcal{L}_{\boxto})$ is inconsistent iff, for some $m,n\in \omega$ some $\phi_1,\ldots,\phi_n\in \Gamma$ and some $\psi_1,\ldots,\psi_m\in \Delta$ we have: $
	\bigwedge^n_{i = 1}\phi_i\vdash\bigvee^m_{j = 1}\psi_j$, or, equivalently, $\vdash
	\bigwedge^n_{i = 1}\phi_i\to\bigvee^m_{j = 1}\psi_j$.
\end{lemma}
The next two lemmas present some properties of the consistent and maximal bi-sets, respectively:
\begin{lemma}\label{L:consistent}
	Let	$(\Gamma, \Delta)\in \mathcal{P}(\mathcal{L}_{\boxto})\times\mathcal{P}(\mathcal{L}_{\boxto})$ be consistent. Then the following statements hold:
	\begin{enumerate}
		\item For every $\phi \in \mathcal{L}_{\boxto}$, either $(\Gamma \cup \{\phi\}, \Delta)$ or $(\Gamma, \Delta\cup \{\phi\})$ is consistent.
		
		\item For every $\phi \to \psi \in \Delta$, $(\Gamma \cup \{\phi\}, \{\psi\})$ is consistent.
		
		\item For every $\phi\boxto\psi \in \Delta$, the bi-set $(\{\chi\mid\phi\boxto\chi \in \Gamma\},\{\psi\})$ is consistent.
		
		\item For every $\sim(\phi\boxto\psi) \in \Gamma$, the bi-set $(\{\sim\psi\}\cup\{\chi\mid\phi\boxto\chi \in \Gamma\},\{\sim\theta\mid\sim(\phi\boxto\theta)\in\Delta\})$ is consistent.
	\end{enumerate}
\end{lemma}
\begin{proof}
Parts 1 and 2 are proved as in the case of $\mathsf{N4}$ (in which respect $\mathsf{N4}$ also just repeats the similar reasoning for $\mathsf{IL}$). As for Part 3, assume that  $\phi\boxto\psi \in \Delta$, and assume, towards  contradiction, that the bi-set $(\{\chi\mid\phi\boxto\chi \in \Gamma\},\{\psi\})$ is inconsistent. Then there must be $\phi\boxto\chi_1,\ldots,\phi\boxto\chi_n \in \Gamma$ such that, for $\chi:= \bigwedge^n_{i= 1}\chi_i$, we have $\chi\vdash \psi$, hence also $\vdash\chi\to\psi$ by Lemma \ref{L:n4}.1 and $\vdash(\phi\boxto\chi)\to(\phi\boxto\psi)$ by \eqref{E:Rmbox}. Next, \eqref{E:a1} implies that $\Gamma\vdash\phi\boxto\chi$, whence also $\Gamma\vdash\phi\boxto\psi$. But then  the assumption that $\phi\boxto\psi \in \Delta$ clearly contradicts the consistency of $(\Gamma,\Delta)$. The obtained contradiction shows that $(\{\chi\mid\phi\boxto\chi \in \Gamma\},\{\psi\})$ must be consistent.

Finally, as for Part 4, assume that $\sim(\phi\boxto\psi) \in \Gamma$, and assume, towards contradiction, that the bi-set $(\{\sim\psi\}\cup\{\chi\mid\phi\boxto\chi \in \Gamma\},\{\sim\theta\mid\sim(\phi\boxto\theta)\in\Delta\})$ is inconsistent. Then there must be some $\phi\boxto\chi_1,\ldots,\phi\boxto\chi_n \in \Gamma$ and $\sim(\phi\boxto\theta_1),\ldots,\sim(\phi\boxto\theta_m) \in \Delta$ such that, for $\chi:= \bigwedge^n_{i= 1}\chi_i$ and $\theta:= \bigwedge^m_{j= 1}\theta_j$,  we have $\chi,\sim\psi\vdash \sim\theta$. We then reason as follows:
\begin{align}
	&\Gamma\vdash\phi\boxto\chi\label{E:cons4}&&\text{by \eqref{E:a1}}\\
	&\vdash\chi\to(\sim\psi\to \sim\theta)\label{E:cons5}&&\text{by Lemma \ref{L:n4}.1}\\
	&\vdash(\phi\boxto\chi)\to(\phi\boxto(\sim\psi\to \sim\theta))\label{E:cons6}&&\text{by \eqref{E:cons5}, \eqref{E:Rmbox}}\\
	&\Gamma\vdash\sim(\phi\boxto\psi)\to \sim(\phi\boxto\theta)\label{E:cons8}&&\text{by \eqref{E:cons4},\eqref{E:cons6},\eqref{E:T2}}\\
	&\Gamma\vdash\sim(\phi\boxto\theta)\label{E:cons9}&&\text{by \eqref{E:cons8},$\sim(\phi\boxto\psi) \in \Gamma$}\\
	&\Gamma\vdash\sim(\phi\boxto\theta_1\wedge\ldots\wedge\phi\boxto\theta_m)\label{E:cons10}&&\text{by \eqref{E:cons9}, \eqref{E:a1}}\\		&\Gamma\vdash\sim(\phi\boxto\theta_1)\vee\ldots\vee\sim(\phi\boxto\theta_m)\label{E:cons11}&&\text{by \eqref{E:cons10}, \eqref{E:a0.2}}
\end{align}
It follows now from \eqref{E:cons11} that $(\Gamma,\Delta)$ must be inconsistent, which contradicts our initial assumption. The obtained contradiction shows that the bi-set $(\{\sim\psi\}\cup\{\chi\mid\phi\boxto\chi \in \Gamma\},\{\sim\theta\mid\sim(\phi\boxto\theta)\in\Delta\})$ must be, in fact, consistent. 
\end{proof} 
\begin{lemma}\label{L:maximal}
	Let	$(\Gamma, \Delta), (\Gamma_0,\Delta_0), (\Gamma_1,\Delta_1) \in \mathcal{P}(\mathcal{L}_{\boxto})\times\mathcal{P}(\mathcal{L}_{\boxto})$ be maximal, let $\phi,\psi\in\mathcal{L}_{\boxto}$. Then the following statements are true:
	\begin{enumerate}
		\item If $\Gamma\vdash\phi$, then $\phi\in \Gamma$.
		
		\item $\phi\wedge\psi\in\Gamma$ iff $\phi, \psi\in \Gamma$.
		
		\item $\phi\vee\psi \in \Gamma$ iff $\phi \in \Gamma$ or $\psi\in\Gamma$.
		
		\item If $\phi\to\psi, \phi \in \Gamma$, then $\psi \in \Gamma$.
		
	\item $\sim(\phi\wedge\psi)\in\Gamma$ iff $\sim\phi\in \Gamma$ or $\sim\psi\in \Gamma$.
	
	\item $\sim(\phi\vee\psi) \in \Gamma$ iff $\sim\phi, \sim\psi\in\Gamma$.
	
	\item $\sim(\phi\to\psi) \in \Gamma$ iff $\phi, \sim\psi \in \Gamma$.
		
		\item If $\Gamma_0 \subseteq \Gamma$, $\{\psi\mid \phi\boxto\psi\in \Gamma_0\}\subseteq \Gamma_1$, and $\{\sim(\phi\boxto\psi)\mid \sim\psi\in \Gamma_1\}\subseteq \Gamma_0$ then
		 
		\noindent$(\Gamma_1 \cup \{\psi\mid\phi\boxto\psi\in \Gamma\}, \{\sim\psi\mid\sim(\phi\boxto\psi)\in \Delta\})$ is consistent.
		
		\item If $\Gamma_1 \subseteq \Gamma$, $\{\psi\mid \phi\boxto\psi\in \Gamma_0\}\subseteq \Gamma_1$, and $\{\sim(\phi\boxto\psi)\mid \sim\psi\in \Gamma_1\}\subseteq \Gamma_0$ then
		 
		\noindent$(\Gamma_0 \cup \{\sim(\phi\boxto\psi)\mid\sim\psi\in \Gamma\}, \{\phi\boxto\psi\mid\psi\in \Delta\})$ is consistent.
	\end{enumerate}
\end{lemma}
\begin{proof}
	The Parts 1--7 are handled as in the case of $\mathsf{N4}$. E.g., for Part 3 observe that, if $\phi\vee\psi \in \Gamma$ and $\phi,\psi\in \Delta$, then we must have $\Gamma\vdash\phi\vee\psi$, thus contradicting the consistency of $(\Gamma, \Delta)$.
	
	As for Part 8, assume its hypothesis and suppose, towards contradiction that $(\Gamma_1 \cup \{\psi\mid\phi\boxto\psi\in \Gamma\}, \{\sim\psi\mid\sim(\phi\boxto\psi)\in \Delta\})$ is inconsistent. Then there must exist some $\psi_1,\ldots,\psi_n \in \Gamma_1$, $\phi\boxto\chi_1,\ldots,\phi\boxto\chi_m\in \Gamma$ and some $\sim(\phi\boxto\theta_1),\ldots,\sim(\phi\boxto\theta_k)\in \Delta$, such that, for $\psi:= \bigwedge^n_{i=1}\psi_i$, $\chi:= \bigwedge^m_{j=1}\chi_j$, and $\theta:= \bigwedge^k_{r=1}\theta_r$ we have $\psi,\chi\vdash\sim\theta$. But then, by Lemma \ref{L:n4}.1, $\psi\vdash\chi\to\sim\theta$, whence, by Part 1, $\chi\to\sim\theta$ (and thus also $\sim\sim(\chi\to\sim\theta)$) must be in $\Gamma_1$. We now reason as follows:
	\begin{align}
		&\sim(\phi\boxto\sim(\chi\to\sim\theta))\in \Gamma_0\label{E:max7}&& \{\sim(\phi\boxto\psi)\mid \sim\psi\in \Gamma_1\}\subseteq \Gamma_0\\
		&\sim(\phi\boxto\sim(\chi\to\sim\theta))\in \Gamma\label{E:max8}&&\text{\eqref{E:max7}, $\Gamma_0\subseteq \Gamma$}\\
	&\Gamma\vdash (\phi\boxto\chi)\to\sim(\phi\boxto\theta)\label{E:max10}&&\text{\eqref{E:max8}, \eqref{E:T3}}\\
	&\Gamma\vdash\phi\boxto\chi\label{E:max11}&&\text{\eqref{E:a1}}\\
	&\Gamma\vdash\sim(\phi\boxto\theta)\label{E:max12}&&\text{\eqref{E:max10}, \eqref{E:max11}}\\
	&\Gamma\vdash\sim(\phi\boxto\theta_1\wedge\ldots\wedge\phi\boxto\theta_m)\label{E:max12a}&&\text{by \eqref{E:max12}, \eqref{E:a1}}\\		&\Gamma\vdash\sim(\phi\boxto\theta_1)\vee\ldots\vee\sim(\phi\boxto\theta_m)\label{E:max13}&&\text{by \eqref{E:max12a}, \eqref{E:a0.2}}
\end{align}
It follows now from \eqref{E:max13}, that $(\Gamma, \Delta)$ is not consistent and thus also not maximal, contrary to our initial assumption. The obtained contradiction shows that the bi-set 

\noindent$(\Gamma_1 \cup \{\psi\mid\phi\boxto\psi\in \Gamma\}, \{\sim\psi\mid\sim(\phi\boxto\psi)\in \Delta\})$ must have been consistent.

For Part 9,  assume its hypothesis and suppose that $(\Gamma_0 \cup \{\sim(\phi\boxto\psi)\mid\sim\psi\in \Gamma\}, \{\phi\boxto\psi\mid\psi\in \Delta\})$ is inconsistent. Then there must exist some $\psi_1,\ldots,\psi_n \in \Gamma_0$, $\sim\chi_1,\ldots,\sim\chi_m\in \Gamma$, and some $\theta_1,\ldots,\theta_k\in \Delta$, such that
$\bigwedge^n_{i=1}\psi_i,\bigwedge^m_{j = 1}\sim(\phi\boxto\chi_j)\vdash\bigvee^k_{r = 1}(\phi\boxto\theta_r)$.
Again, we set $\psi:= \bigwedge^n_{i=1}\psi_i$, $\chi:= \bigwedge^m_{j=1}\sim\chi_j$, and $\theta:= \bigvee^k_{r=1}\theta_r$, and reason as follows:
\begin{align}
	&\Gamma_0\vdash\bigwedge^m_{j = 1}\sim(\phi\boxto\chi_j)\to\bigvee^k_{r = 1}(\phi\boxto\theta_r)\label{E:maxi4}&&\text{Lemma \ref{L:n4}.1}\\
	&\vdash\sim(\phi\boxto\sim\chi)\to\bigwedge^m_{j = 1}\sim(\phi\boxto\chi_j)\label{E:maxi7}&&\text{$\mathsf{N4}$, \eqref{E:Rnbox}}\\
	&\vdash\bigvee^k_{r = 1}(\phi\boxto\theta_j)\to(\phi\boxto \theta)\label{E:maxi9}&&\text{$\mathsf{N4}$, \eqref{E:Rmbox}}\\
	&\Gamma_0\vdash\sim(\phi\boxto\sim\chi)\to(\phi\boxto \theta)\label{E:maxi10}&&\text{\eqref{E:maxi4},\eqref{E:maxi7},\eqref{E:maxi9}}\\
	&\phi\boxto(\chi\to\theta)\in\Gamma_0\label{E:maxi12}&&\text{\eqref{E:maxi10},\eqref{E:a3}, Part 1}\\
	&(\chi\to\theta)\in\Gamma_1\label{E:maxi13}&&\text{\eqref{E:maxi12}, $\{\psi\mid \phi\boxto\psi\in \Gamma_0\}\subseteq \Gamma_1$}
\end{align}
By $\Gamma_1\subseteq \Gamma$, we know then that also $(\chi\to\theta)\in\Gamma$. Now, since clearly $\Gamma\vdash\chi$, it follows that $\Gamma\vdash\theta = \bigvee^k_{r=1}\theta_r$, which clearly contradicts the consistency of $(\Gamma, \Delta)$. The obtained contradiction shows that the bi-set $(\Gamma_0 \cup \{\sim(\phi\boxto\psi)\mid\sim\psi\in \Gamma\}, \{\phi\boxto\psi\mid\psi\in \Delta\})$ must have been consistent.
\end{proof}
We observe, next, that we can use the usual Lindenbaum construction to extend every consistent bi-set to a maximal one:
\begin{lemma}\label{L:lindenbaum}
	Let $(\Gamma, \Delta)\in \mathcal{P}(\mathcal{L}_{\boxto})\times\mathcal{P}(\mathcal{L}_{\boxto})$ be consistent. Then there exists a maximal $(\Xi, \Theta)\in \mathcal{P}(\mathcal{L}_{\boxto})\times\mathcal{P}(\mathcal{L}_{\boxto})$ such that $\Gamma \subseteq \Xi$ and $\Delta \subseteq \Theta$. 
\end{lemma}
Next, we define the canonical model $\mathcal{M}_c$ for $\mathbb{N}4\mathbb{CK}$:
\begin{definition}\label{D:canonical-model}
	The structure $\mathcal{M}_c$ is the tuple $(W_c, \leq_c, R_c, V^+_c, V^-_c)$ such that:
	\begin{itemize}
		\item $W_c:=\{(\Gamma, \Delta)\in \mathcal{P}(\mathcal{L})\times\mathcal{P}(\mathcal{L})\mid (\Gamma, \Delta)\text{ is maximal}\}$.
		
		\item $(\Gamma_0,\Delta_0)\leq_c(\Gamma_1,\Delta_1)$ iff $\Gamma_0\subseteq\Gamma_1$ for all $(\Gamma_0,\Delta_0),(\Gamma_1,\Delta_1)\in W_c$.
		
		\item For all $(\Gamma_0,\Delta_0),(\Gamma_1,\Delta_1)\in W_c$ and $X \subseteq W_c$, we have $((\Gamma_0,\Delta_0),(X,Y),(\Gamma_1,\Delta_1)) \in R_c$ iff there exists a $\phi\in\mathcal{L}_{\boxto}$, such that all of the following holds:
		\begin{itemize}
			\item $X = \{(\Gamma,\Delta)\in W_c\mid\phi\in\Gamma\}$.
			
			\item $Y = \{(\Gamma,\Delta)\in W_c\mid\sim\phi\in\Gamma\}$
			
			\item $\{\psi\mid\phi\boxto\psi\in \Gamma_0\}\subseteq \Gamma_1$.
			
			\item $\{\sim(\phi\boxto\psi)\mid\sim\psi\in \Gamma_1\}\subseteq \Gamma_0$.
		\end{itemize}
		
		\item $V^+_c(p):=\{(\Gamma,\Delta)\in W_c\mid p\in\Gamma\}$ for every $p \in Var$.
		
		\item $V^-_c(p):=\{(\Gamma,\Delta)\in W_c\mid \sim p\in\Gamma\}$ for every $p \in Var$.	
	\end{itemize}
\end{definition}
First of all, we observe that the definition of $R_c$ does not depend on the choice of the representative formula $\phi\in\mathcal{L}_{\boxto}$. The following lemma provides the necessary stepping stone:
\begin{lemma}\label{L:representatives}
	Let $\phi,\psi \in \mathcal{L}$ be such that both $\{(\Gamma,\Delta)\in W_c\mid\phi\in\Gamma\} = \{(\Gamma,\Delta)\in W_c\mid\psi\in\Gamma\}$ and $\{(\Gamma,\Delta)\in W_c\mid\sim\phi\in\Gamma\} = \{(\Gamma,\Delta)\in W_c\mid\sim\psi\in\Gamma\}$. Then, for every $(\Gamma',\Delta')\in W_c$ and every $\chi\in\mathcal{L}_{\boxto}$ we have:
	\begin{enumerate}
		\item $\phi\boxto\chi\in\Gamma'$ iff $\psi\boxto\chi\in\Gamma'$.
		
		\item $\sim(\phi\boxto\chi)\in\Gamma'$ iff $\sim(\psi\boxto\chi)\in\Gamma'$.
	\end{enumerate}
\end{lemma}
\begin{proof}
	Assume the hypothesis of the Lemma. We will show that in this case we must have $\vdash\phi\Leftrightarrow\psi$. Suppose not, and assume, for instance, that $\not\vdash\phi\to\psi$. Then $(\{\phi\},\{\psi\})$ must be consistent and thus extendable to some maximal $(\Gamma_0,\Delta_0)\supseteq (\{\phi\},\{\psi\})$. But then clearly $
	(\Gamma_0,\Delta_0)\in \{(\Gamma,\Delta)\in W_c\mid\phi\in\Gamma\} \setminus \{(\Gamma,\Delta)\in W_c\mid\psi\in\Gamma\}$, 
	in contradiction with our initial assumptions. On the other hand, if we have, e.g. $\not\vdash\sim\phi\to\sim\psi$, then $(\{\sim\phi\},\{\sim\psi\})$ must be consistent and thus extendable to some maximal $(\Gamma_1,\Delta_1)\supseteq (\{\sim\phi\},\{\sim\psi\})$. But then clearly
	$$
	(\Gamma_1,\Delta_1)\in \{(\Gamma,\Delta)\in W_c\mid\sim\phi\in\Gamma\} \setminus \{(\Gamma,\Delta)\in W_c\mid\sim\psi\in\Gamma\},
	$$ 
	again contradicting our initial assumptions. The reasoning in other cases is parallel to the examples considered above.
	
	Thus we see that we must have $\vdash\phi\Leftrightarrow\psi$. An application of \eqref{E:RAbox} then yields that also $\vdash(\phi\boxto\chi)\Leftrightarrow(\psi\boxto\chi)$
	for every $\chi\in\mathcal{L}_{\boxto}$, whence our Lemma clearly follows.  
\end{proof}
We have to make sure that we have indeed just defined a model:
\begin{lemma}\label{L:canonical-model}
$\mathcal{M}_c \in CNel$.	
\end{lemma}
\begin{proof}
	We show, first, that $W_c \neq \emptyset$. Indeed, choose any $p \in Var$ and consider the bi-set $(\emptyset,\{p\})$. Clearly, $p\in\mathcal{L}\setminus\mathsf{N4}$, whence Lemma \ref{L:n4} implies that $(\emptyset,\{p\})$ is consistent. Therefore, by Lemma \ref{L:lindenbaum}, there must exist a maximal $(\Gamma, \Delta)\supseteq(\emptyset,\{p\})$; and we will have, by Definition \ref{D:canonical-model}, that $(\Gamma, \Delta)\in W_c$.
	
	It is also clear from Definition  \ref{D:canonical-model} and Lemma \ref{L:representatives} that $\leq_c$ is a pre-order, and that $R_c\subseteq W_c\times(\mathcal{P}(W_c)\times\mathcal{P}(W_c))\times W_c$ is well-defined. So it only remains to check the satisfaction of conditions \eqref{Cond:1} and \eqref{Cond:2} from Definition \ref{D:model}.
	
	As for \eqref{Cond:1}, assume that $(\Gamma,\Delta)$,  $(\Gamma_0,\Delta_0)$, and  $(\Gamma_1,\Delta_1)$ are maximal, and that $X, Y\subseteq W_c$ are such that we have $(\Gamma,\Delta) \mathrel{\geq_c}(\Gamma_0,\Delta_0)\mathrel{(R_c)_{(X,Y)}}(\Gamma_1,\Delta_1)$. Then, in particular, $\Gamma\supseteq\Gamma_0$. Moreover, we can choose a $\phi \in\mathcal{L}_{\boxto}$ such that we have:
	\begin{align}
		&X = \{(\Xi,\Theta)\in W_c\mid\phi\in\Xi\}\label{E:mod1}\\
		&Y = \{(\Xi,\Theta)\in W_c\mid\sim\phi\in\Xi\}\label{E:mod1a}\\
		&\{\psi\mid\phi\boxto\psi\in \Gamma_0\}\subseteq \Gamma_1\label{E:mod2}\\
		&\{\sim(\phi\boxto\psi)\mid\sim\psi\in \Gamma_1\}\subseteq \Gamma_0\label{E:mod3}
	\end{align}
By Lemma \ref{L:maximal}.8, the bi-set $(\Gamma_1 \cup \{\psi\mid\phi\boxto\psi\in \Gamma\}, \{\sim\psi\mid\sim(\phi\boxto\psi)\in \Delta\})$ must then be consistent, so that, by Lemma \ref{L:lindenbaum}, this bi-set must be extendable to some maximal bi-set $(\Gamma',\Delta')$. We will have then $\Gamma'\supseteq \Gamma_1$ whence clearly $(\Gamma',\Delta')\mathrel{\geq_c}(\Gamma_1,\Delta_1)$.

Next, we get $\{\psi\mid\phi\boxto\psi\in \Gamma\}\subseteq \Gamma'$ trivially by the choice of $(\Gamma',\Delta')$. Moreover, if $\sim\psi \in \Gamma'$, then $\sim\psi\notin\Delta'$ by the consistency of $(\Gamma',\Delta')$. But this means that we cannot have $\sim(\phi\boxto\psi)\in\Delta$, so $\sim(\phi\boxto\psi)\in \Gamma$ by the completeness of $(\Gamma,\Delta)$. Thus we have shown that also $\{\sim(\phi\boxto\psi)\mid\sim\psi\in \Gamma'\}\subseteq \Gamma$. Summing this up with \eqref{E:mod1} and \eqref{E:mod1a}, we obtain that $(\Gamma,\Delta)\mathrel{(R_c)_{(X,Y)}}(\Gamma',\Delta')$. Thus we get that $
(\Gamma,\Delta)\mathrel{(R_c)_{(X,Y)}}(\Gamma',\Delta')\mathrel{\geq_c}(\Gamma_1,\Delta_1)$, and condition \eqref{Cond:1} is shown to be satisfied.

As for \eqref{Cond:2}, assume that $(\Gamma,\Delta)$,  $(\Gamma_0,\Delta_0)$, and  $(\Gamma_1,\Delta_1)$ are maximal, and that $X,Y\subseteq W_c$ are such that we have $(\Gamma_0,\Delta_0)\mathrel{(R_c)_{(X,Y)}} (\Gamma_1,\Delta_1)\mathrel{\leq_c}(\Gamma,\Delta)$. Then $\Gamma\supseteq\Gamma_1$. Moreover, we can choose a $\phi \in\mathcal{L}_{\boxto}$ such that all of \eqref{E:mod1}--\eqref{E:mod3} hold.

By Lemma \ref{L:maximal}.9, the bi-set $(\Gamma_0 \cup \{\sim(\phi\boxto\psi)\mid\sim\psi\in \Gamma\}, \{\phi\boxto\psi\mid\psi\in \Delta\})$ must then be consistent, so that, by Lemma \ref{L:lindenbaum}, this bi-set must be extendable to some maximal bi-set $(\Gamma',\Delta')$. Clearly, $\Gamma'\supseteq \Gamma_0$ whence also $(\Gamma',\Delta')\mathrel{_c\geq}(\Gamma_0,\Delta_0)$.

Next, assume that $\psi\in \mathcal{L}_{\boxto}$ is such that $\phi\boxto\psi\in \Gamma'$. If $\psi\notin\Gamma$, then $\psi \in\Delta$ by the completeness of $(\Gamma,\Delta)$, whence $\phi\boxto\psi\in \Delta'$. But the latter contradicts the consistency of $(\Gamma',\Delta')$. Therefore $\psi\in\Gamma$. Since the choice of $\psi$ was arbitrary, we have shown that $\{\psi\mid\phi\boxto\psi\in \Gamma'\}\subseteq \Gamma$. Moreover, if $\sim\psi\in\Gamma$, then $\sim(\phi\boxto\psi) \in \Gamma'$ so that $\{\sim(\phi\boxto\psi)\mid\sim\psi\in \Gamma\}\subseteq \Gamma'$ also holds. Summing this up with \eqref{E:mod1} and \eqref{E:mod1a}, we obtain that $(\Gamma',\Delta')\mathrel{(R_c)_{(X,Y)}}(\Gamma,\Delta)$. Thus we get that $(\Gamma_0,\Delta_0)\mathrel{\leq_c}(\Gamma',\Delta')\mathrel{(R_c)_{(X,Y)}}(\Gamma,\Delta)$, and condition \eqref{Cond:2} is shown to be satisfied.
\end{proof}
The truth lemma for this model then looks as follows:
\begin{lemma}\label{L:truth}
	For every $\phi\in\mathcal{L}_{\boxto}$ and for every $(\Gamma,\Delta)\in W_c$, the following statements hold:
	\begin{enumerate}
		\item $
		\mathcal{M}_c,(\Gamma,\Delta)\models^+\phi$ iff $\phi \in \Gamma$.
		
		\item $
		\mathcal{M}_c,(\Gamma,\Delta)\models^-\phi$ iff $\sim\phi \in \Gamma$.
	\end{enumerate} 
\end{lemma}
\begin{proof}
	We prove both parts by simultaneous induction on the construction of $\phi$.
	
\textit{Basis}. If $\phi = p \in Var$, then the lemma holds by the definition of $\mathcal{M}_c$.

\textit{Induction step}. The cases associated with $\wedge$, $\vee$, $\sim$, and $\to$ are solved as in the case of $\mathsf{N4}$. We treat the case when $\phi = \psi \boxto \chi$ in some detail:

\textit{Part 1}. ($\Leftarrow$)  Let $(\Gamma,\Delta)\in W_c$ be such that $\phi = \psi \boxto \chi \in \Gamma$, and let $(\Gamma_0,\Delta_0), (\Gamma_1,\Delta_1)\in W_c$ be such that $(\Gamma,\Delta)\mathrel{\leq_c}(\Gamma_0,\Delta_0)\mathrel{(R_c)_{\|\psi\|_{\mathcal{M}_c}}}(\Gamma_1,\Delta_1)$. Then $\Gamma \subseteq \Gamma_0$, so that $\psi \boxto \chi \in \Gamma_0$. On the other hand, there must exist a $\theta\in \mathcal{L}_{\boxto}$ such that all of the following holds:
\begin{align}
	&\|\psi\|_{\mathcal{M}_c} =(\{(\Xi,\Theta)\in W_c\mid\theta\in\Xi\}, \{(\Xi,\Theta)\in W_c\mid\sim\theta\in\Xi\})\label{E:choice1}\\
	&\{\xi\mid\theta\boxto\xi\in \Gamma_0\}\subseteq \Gamma_1\label{E:choice2}\\
	&\{\sim(\theta\boxto\xi)\mid\sim\xi\in \Gamma_1\}\subseteq \Gamma_0\label{E:choice3}
\end{align}
By IH, we know that also $\|\psi\|_{\mathcal{M}_c} = (\{(\Xi,\Theta)\in W_c\mid\psi\in\Xi\}, \{(\Xi,\Theta)\in W_c\mid\sim\psi\in\Xi\})$. We thus get that:
\begin{align}
	(\{(\Xi,\Theta)\in W_c\mid\psi\in\Xi\}, &\{(\Xi,\Theta)\in W_c\mid\sim\psi\in\Xi\}) =\notag\\
	&= (\{(\Xi,\Theta)\in W_c\mid\theta\in\Xi\}, \{(\Xi,\Theta)\in W_c\mid\sim\theta\in\Xi\})\label{E:choice4}
\end{align}
Since $\psi \boxto \chi \in \Gamma_0$, we know, by Lemma \ref{L:representatives} and \eqref{E:choice4}, that also $\theta\boxto\chi \in \Gamma_0$. It follows by \eqref{E:choice2}, that $\chi\in \Gamma_1$. Next, IH implies that  $\mathcal{M}_c,(\Gamma_1,\Delta_1)\models^+\chi$. Since the choice of  $(\Gamma_0,\Delta_0), (\Gamma_1,\Delta_1)\in W_c$ under the condition that $(\Gamma,\Delta)\mathrel{\leq_c}(\Gamma_0,\Delta_0)\mathrel{(R_c)_{\|\psi\|_{\mathcal{M}_c}}}(\Gamma_1,\Delta_1)$ was made arbitrarily, it follows that we must have $\mathcal{M}_c,(\Gamma,\Delta)\models^+ \psi \boxto \chi = \phi$.

($\Rightarrow$)  Let $(\Gamma,\Delta)\in W_c$ be such that $\phi \notin \Gamma$. Therefore, $\psi \boxto \chi \in \Delta$ by completeness of $(\Gamma, \Delta)$, and $(\{\theta\mid\psi\boxto\theta\in\Gamma\},\{\chi\})$ must be consistent by Lemma \ref{L:consistent}.3. By Lemma \ref{L:lindenbaum}, we can extend it to a maximal $(\Gamma',\Delta')\supseteq (\{\theta\mid\psi\boxto\theta\in\Gamma\},\{\chi\})$. Now, set $(\Gamma_0,\Delta_0):= (\Gamma\cup\{\sim(\psi\boxto\theta)\mid\sim\theta\in \Gamma'\}, \{\psi\boxto\xi\mid\xi\in\Delta'\})$. We claim that $(\Gamma_0,\Delta_0)$ is consistent. Otherwise, we can choose $\gamma_1,\ldots,\gamma_n\in\Gamma$, $\sim\tau_1,\ldots,\sim\tau_m\in\Gamma'$ and $\xi_1,\ldots,\xi_k\in\Delta'$ such that $\bigwedge^n_{i = 1}\gamma_i,\bigwedge^m_{j = 1}\sim(\psi\boxto\tau_j)\vdash\bigvee^k_{r = 1}(\psi\boxto\xi_r)$. But then, for $\gamma:= \bigwedge^n_{i = 1}\gamma_i$, $\tau:= \bigwedge^m_{j = 1}\sim\tau_j$, and $\xi:= \bigvee^k_{r = 1}\xi_r$, we have:
\begin{align}
	&\gamma\vdash \bigwedge^m_{j = 1}\sim(\psi\boxto\tau_j)\to\bigvee^k_{r = 1}(\psi\boxto\xi_r)\label{E:can5} &&\text{Lemma \ref{L:n4}.1}\\
	&\vdash\sim(\psi\boxto\sim\tau)\to\bigwedge^m_{j = 1}\sim(\psi\boxto\tau_j)\label{E:can7} &&\text{$\mathsf{N4}$, \eqref{E:Rnbox}}\\
	&\vdash\bigvee^k_{r = 1}(\psi\boxto\xi_r)\to(\psi\boxto\xi)\label{E:can9} &&\text{$\mathsf{N4}$, \eqref{E:Rmbox}}\\
	&\gamma\vdash(\psi\diamondto\tau)\to(\psi\boxto\xi)\label{E:can10} &&\text{\eqref{E:can5}, \eqref{E:can7}, \eqref{E:can9}}\\
	&(\psi\boxto(\tau\to\xi))\in \Gamma\label{E:can12} &&\text{\eqref{E:can10}, \eqref{E:a4}, Lemma \ref{L:maximal}.1}
\end{align}
By \eqref{E:can12} and the choice of $(\Gamma',\Delta')$, we know that $(\tau\to\xi)\in \Gamma'$, therefore, $(\Gamma',\Delta')$ must be inconsistent, which contradicts its choice and shows that $(\Gamma_0,\Delta_0)$ must have been consistent. Therefore, $(\Gamma_0,\Delta_0)$ is extendable to a maximal bi-set $(\Gamma_1,\Delta_1)\supseteq(\Gamma_0,\Delta_0)$. 

We now claim that we have both $(\Gamma,\Delta)\mathrel{\leq_c}(\Gamma_1,\Delta_1)$ and 
$$
((\Gamma_1,\Delta_1),(\{(\Xi,\Theta)\in W_c\mid\psi\in\Xi\}, \{(\Xi,\Theta)\in W_c\mid\sim\psi\in\Xi\}),(\Gamma',\Delta'))\in R_c.
$$
The first part is trivial since we have $\Gamma_1\supseteq\Gamma_0 \supseteq\Gamma$ by the choice of $(\Gamma_1,\Delta_1)$ and $(\Gamma_0,\Delta_0)$. As for the second part, note that (a) for every $\theta\in\mathcal{L}_{\boxto}$, if $\psi\boxto\theta\in \Gamma_1$ and $\theta \notin \Gamma'$, then, by the completeness of $(\Gamma',\Delta')$, we must have $\theta \in \Delta'$. But then $\psi\boxto\theta\in\Delta_0\subseteq\Delta_1$, which contradicts the consistency  of $(\Gamma_1,\Delta_1)$. The obtained contradiction shows that  $\{\theta\mid\psi\boxto\theta\in \Gamma_1\}\subseteq \Gamma'$. Next, (b) we trivially get that $\{\sim(\psi\boxto\theta)\mid\sim\theta\in \Gamma'\}\subseteq\Gamma_0 \subseteq \Gamma_1$. Summing up (a) and (b), we get that $((\Gamma_1,\Delta_1),(\{(\Xi,\Theta)\in W_c\mid\psi\in\Xi\},\{(\Xi,\Theta)\in W_c\mid\sim\psi\in\Xi\}),(\Gamma',\Delta'))\in R_c$.

It remains to notice that, by IH, we must have 
$$
\|\psi\|_{\mathcal{M}_c} = (\{(\Xi,\Theta)\in W_c\mid\psi\in\Xi\},\{(\Xi,\Theta)\in W_c\mid\sim\psi\in\Xi\}),
$$
so that we have shown, in effect that $((\Gamma_1,\Delta_1),\|\psi\|_{\mathcal{M}_c},(\Gamma',\Delta'))\in R_c$.

Observe, next, that also $\chi\in \Delta'$, whence $\chi\notin\Gamma'$ by the consistency of $(\Gamma',\Delta')$. Therefore, by IH, $\mathcal{M}_c,(\Gamma',\Delta')\not\models^+ \chi$. Together with the fact that $(\Gamma,\Delta)\mathrel{\leq_c}(\Gamma_1,\Delta_1)$ and $((\Gamma_1,\Delta_1),\|\psi\|_{\mathcal{M}_c},(\Gamma',\Delta'))\in R_c$, this finally implies that $\mathcal{M}_c,(\Gamma',\Delta')\not\models^+ \psi\boxto\chi = \phi$.

\textit{Part 2}. ($\Leftarrow$)  Let $(\Gamma,\Delta)\in W_c$ be such that $\sim\phi = \sim(\psi \boxto \chi) \in \Gamma$. Then, by Lemma \ref{L:consistent}.4, the bi-set $(\{\sim\chi\}\cup\{\xi\mid\psi\boxto\xi \in \Gamma\},\{\sim\theta\mid\sim(\psi\boxto\theta)\in\Delta\})$ must be consistent, and, by Lemma \ref{L:lindenbaum}, there must be a maximal bi-set $(\Gamma',\Delta')\in W_c$ such that $
(\Gamma',\Delta')\supseteq (\{\sim\chi\}\cup\{\xi\mid\psi\boxto\xi \in \Gamma\},\{\sim\theta\mid\sim(\psi\boxto\theta)\in\Delta\})$.

Since $\sim\chi\in\Gamma'$, we must have $\mathcal{M}_c,(\Gamma',\Delta')\models^- \chi$ by IH. On the other hand, IH yields that  $\|\psi\|_{\mathcal{M}_c} = (\{(\Xi,\Theta)\in W_c\mid\psi\in\Xi\},\{(\Xi,\Theta)\in W_c\mid\sim\psi\in\Xi\})$. Next, by the choice of  $(\Gamma',\Delta')$, we know that $\{\xi\mid\phi\boxto\xi \in \Gamma\}\subseteq \Gamma'$. Finally, if $\sim\theta\in\Gamma'$, then, by the consistency of $(\Gamma',\Delta')$, $\sim\theta\notin\Delta'$, whence clearly $\sim(\psi\boxto\theta)\notin\Delta$. But then, by the completeness of  $(\Gamma,\Delta)$, we must have $\sim(\psi\boxto\theta)\in\Gamma$. We have thus shown that $\{\sim(\psi\boxto\theta)\mid\sim\theta\in\Gamma'\}\subseteq\Gamma$. Summing up, we must have $((\Gamma,\Delta),\|\psi\|_{\mathcal{M}_c},(\Gamma',\Delta'))\in R_c$, and, since we have also shown that $\mathcal{M}_c,(\Gamma',\Delta')\models^- \chi$, $\mathcal{M}_c,(\Gamma,\Delta)\models^- \psi\boxto\chi$ clearly follows.

($\Rightarrow$) Let $(\Gamma,\Delta)\in W_c$ be such that $\sim\phi = \sim(\psi \boxto \chi) \notin \Gamma$. Assume now that  $(\Gamma',\Delta')\in W_c$ is such that $((\Gamma,\Delta),\|\psi\|_{\mathcal{M}_c},(\Gamma',\Delta'))\in R_c$. Let $\theta\in\mathcal{L}_{\boxto}$ be such that all of the following holds:
\begin{align}
	&\|\psi\|_{\mathcal{M}_c} = (\{(\Xi,\Theta)\in W_c\mid\theta\in\Xi\},\{(\Xi,\Theta)\in W_c\mid\sim\theta\in\Xi\})\label{E:choice5}\\
	&\{\xi\mid\theta\boxto\xi\in \Gamma\}\subseteq \Gamma'\label{E:choice6}\\
	&\{\sim(\theta\boxto\xi)\mid\sim\xi\in \Gamma'\}\subseteq \Gamma\label{E:choice7}
\end{align}
By IH, we know that also $\|\psi\|_{\mathcal{M}_c} = (\{(\Xi,\Theta)\in W_c\mid\psi\in\Xi\},\{(\Xi,\Theta)\in W_c\mid\sim\psi\in\Xi\})$. We thus get that:
\begin{align}
	(\{(\Xi,\Theta)\in W_c\mid\psi\in\Xi\},&\{(\Xi,\Theta)\in W_c\mid\sim\psi\in\Xi\}) =\notag\\
	&= (\{(\Xi,\Theta)\in W_c\mid\theta\in\Xi\},\{(\Xi,\Theta)\in W_c\mid\sim\theta\in\Xi\})\label{E:choice8}
\end{align}
Since we have $\sim(\psi \boxto \chi) \notin \Gamma$, we know that, by Lemma \ref{L:representatives} and \eqref{E:choice8}, we must also have $\sim(\theta\boxto\chi) \notin \Gamma$. It follows now, by \eqref{E:choice7}, that we must have $\sim\chi\notin \Gamma'$. Next, IH implies that  $\mathcal{M}_c,(\Gamma',\Delta')\not\models^-\chi$. Since the choice of $(\Gamma',\Delta')\in W_c$ under the condition that $((\Gamma,\Delta),\|\psi\|_{\mathcal{M}_c},(\Gamma',\Delta'))\in R_c$
 was made arbitrarily, it follows that we must have$\mathcal{M}_c,(\Gamma,\Delta)\not\models^- \psi \boxto \chi = \phi$.
\end{proof}
The truth lemma allows us to deduce the (strong) soundness and completeness of $\mathbb{N}4\mathbb{CK}$ relative to $\mathsf{N4CK}$ in the usual way:
\begin{theorem}\label{T:completeness}
	$\mathsf{N4CK}= \mathbb{N}4\mathbb{CK}[\mathcal{L}_{\boxto}]$. In particular, for every $\phi\in\mathcal{L}_{\boxto}$, $\vdash\phi\Leftrightarrow(\phi\in\mathsf{N4CK})$.
\end{theorem}
\begin{proof}
	Let $(\Gamma, \Delta) \in \mathcal{P}(\mathcal{L}_{\boxto})\times\mathcal{P}(\mathcal{L}_{\boxto})$. We will show that $\Gamma\models\Delta$ iff $\Gamma\vdash\Delta$.
	
	($\Leftarrow$) We argue by contraposition. First, we show the following claim by induction on the length of a derivation $\psi_1,\ldots,\psi_n = \phi$ of $\phi$ from the premises in $\Gamma$:
	
	\textit{Claim}. If $\Gamma\vdash\phi$, then $\Gamma\models \phi$.
	
	If now $\Gamma\vdash\Delta$, then, by definition, $\Delta \neq \emptyset$, and we must have $\Gamma\vdash\psi_1\vee\ldots\vee\psi_n$ for some $\psi_1,\ldots,\psi_n\in \Delta$. But then the Claim implies that $\Gamma \models \psi_1\vee\ldots\vee\psi_n$, whence clearly $\Gamma\models \{\psi_1,\ldots,\psi_n\}$, so that $\Gamma\models\Delta$ holds as well.
	
	($\Rightarrow$) If $\Gamma\not\vdash\Delta$, then, by Lemma \ref{L:lindenbaum}, choose any maximal $(\Gamma',\Delta')\supseteq (\Gamma,\Delta)$. By Lemma \ref{L:truth}, we have $\mathcal{M}_c,(\Gamma',\Delta')\models(\Gamma,\Delta)$.
	
	We have thus shown Theorem \ref{T:completeness}. In particular we have shown that, for every $\phi \in \mathcal{L}$, $\vdash\phi$ iff $(\emptyset,\{\phi\})$ is inconsistent iff $\phi\in\mathsf{IntCK}$.
\end{proof}
As a usual corollary, we obtain the compactness of $\mathsf{N4CK}$ for bi-sets:
\begin{corollary}\label{C:compactness}
	For $(\Gamma,\Delta)\in \mathcal{P}(\mathcal{L}_{\boxto})\times\mathcal{P}(\mathcal{L}_{\boxto})$, we have $\Gamma\not\models\Delta$ iff, for every $(\Gamma',\Delta')\Subset(\Gamma,\Delta)$, $\Gamma'\not\models\Delta'$.
\end{corollary}
\begin{proof}
	The ($\Rightarrow$)-part is straightforward, as for the converse, we argue by contraposition. If $\Gamma\not\models\Delta$, then, by Theorem \ref{T:completeness}, $\Gamma\not\vdash\Delta$, therefore, for some $\psi_1,\ldots,\psi_n\in \Delta$ we have $\Gamma\vdash\psi_1\vee\ldots\vee\psi_n$. Let $\chi_1,\ldots,\chi_m$ be any derivation of $\psi_1\vee\ldots\vee\psi_n$ from the premises in $\Gamma$ and let $\phi_1,\ldots,\phi_k$ be a list of all formulas from $\Gamma$ occurring among $\chi_1,\ldots,\chi_m$. Then $\chi_1,\ldots,\chi_m$ also shows that $\{\phi_1,\ldots,\phi_k\}\vdash\{\psi_1,\ldots,\psi_n\}$ and hence, by Theorem \ref{T:completeness}, that $\{\phi_1,\ldots,\phi_k\}\models\{\psi_1,\ldots,\psi_n\}$. 
\end{proof}
\begin{remark}\label{R:2}
	A similar argument shows that the logic $\mathsf{N4CK}'$, mentioned in Remark \ref{R:1} is axiomatized by extending $\mathbb{N}4\mathbb{CK}$ with the following rule:
	\begin{align}
		\text{From }\phi\leftrightarrow\psi &\text{ infer } (\phi\boxto\chi)\Leftrightarrow(\psi\boxto\chi)\label{E:RA'box}\tag{RA'$\Box$}
	\end{align}	
\end{remark}

\section{Relations with other logics}\label{S:other}
In the existing literature, one can find several systems which can be viewed as natural companions to $\mathsf{N4CK}$. In this paper, we confine ourselves to mentioning but a few prominent examples that fall into two groups: other conditional logics and modal logics. We treat these groups in the two subsections of the present section, and, considering the length of this paper, most of our claims will only be supplied with a rather sketchy proof. 

\subsection{Conditional logics}\label{sub:conditional}

The first of the systems that we would like to consider is the basic system $\mathsf{CK}$ of classical conditional logic, introduced in \cite{chellas} and defined over $\mathcal{L}_{\boxto}$. In particular, it is shown in \cite{chellas} that we have $\mathsf{CK} = (\mathbb{CL}+(\eqref{E:a1},\eqref{E:a4};\eqref{E:RAbox},\eqref{E:RCbox1}))^\iota[\mathcal{L}_{\boxto}]$. 

 It is natural to expect that the relation between $\mathsf{CK}$ and $\mathsf{N4CK}$ is similar to the relation established by Proposition \ref{P:prop}.4 between their respective propositional bases, namely $\mathbb{CL}$ and $\mathbb{N}4$. This is indeed the case, as we will show presently. We prepare the result with a technical lemma:
\begin{lemma}\label{L:CK}
	Every instance of \eqref{E:a2}, \eqref{E:a3}, and \eqref{E:RCbox2} as well as all the theorems and derived rules given in Lemma \ref{L:theorems}, are deducible in $\mathsf{CK}$. 
\end{lemma}
We sketch the proof in Appendix \ref{A:2}.

The relation between $\mathsf{N4CK}$ and $\mathsf{CK}$ can then be formulated as follows:
\begin{proposition}\label{P:CK}
	 The following statements are true for every $\phi\in\mathcal{L}_{\boxto}$:
	 \begin{enumerate}
	 	\item If $\phi\in\mathsf{N4CK}$, then $\phi\in \mathsf{CK}$.
	 	
	 	\item $\mathbb{CK}^\iota[\mathcal{L}_{\boxto}] = (\mathbb{N}4\mathbb{CK}+(\eqref{E:a0.6},\eqref{E:a0.7};))^\iota[\mathcal{L}_{\boxto}]$.
	 \end{enumerate}
 \end{proposition}
 \begin{proof}
 	By Lemma \ref{L:CK}.
 \end{proof}
Yet another logic that is very natural to compare with $\mathsf{N4CK}$ is the extended positive fragment $\mathsf{IntCK}^{e+}$ of the system of intuitionistic conditional logic $\mathsf{IntCK}$ introduced in \cite{olkhovikov}. The language of $\mathsf{IntCK}^{e+}$ is $\mathcal{L}^{e+}_{(\boxto,\diamondto)}$; in other words, $\sim$ is omitted, the additional variables from $Var'$ are included, and $\diamondto$ is no longer an abbreviation, but an elementary connective. 

The Chellas version of Kripke semantics for $\mathsf{IntCK}^{e+}$ is based on the class $CInt^e$ of extended conditional intuitionistic models given by the following definition:
\begin{definition}\label{D:int-model}
	An \textit{extended conditional intuitionistic model} is a structure of the form $\mathcal{M} = (W, \leq, R, V)$, where $(W, \leq, V)\in Int^e$ and  $R \subseteq W \times \mathcal{P}(W)\times W$. Thus, for every $X \subseteq W$, $R$ induces a binary relation $R_X$ on $W$ such that, for all $w,v \in W$, $R_X(w,v)$ iff $R(w, X, v)$. For every $X \subseteq W$, $R$ must satisfy the following conditions:
	\begin{align}
		(\leq^{-1}\circ R_X) &\subseteq (R_X\circ\leq^{-1})\label{Cond:1i}\tag{c1-i}\\
		(R_X\circ\leq) &\subseteq (\leq\circ R_X)\label{Cond:2i}\tag{c2-i}
	\end{align}	
\end{definition}
In case $\mathcal{M} = (W, \leq, R, V) \in CInt^e$, the structure $\mathcal{M} = (W, \leq, R, V\upharpoonright Var)$ is called \textit{conditional intuitionistic model}. We will denote by $CInt$ the class of all such models.

The semantics of $\mathsf{IntCK}^{e+}$ also uses just one satisfaction relation (we will denote it by $\models^i$) in place of the two relations, $\models^+$ and $\models^-$ of $\mathsf{N4CK}$. The inductive definition of $\models^i$ includes every clause given for $\models_{\mathsf{IL}}$ in Section \ref{S:basis} plus the following clauses for the conditional connectives:
\begin{align*}
	\mathcal{M}, w&\models^i \psi \boxto \chi \Leftrightarrow (\forall v \geq w)(\forall u \in W)(R_{\|\psi\|^i_\mathcal{M}}(v, u) \Rightarrow\mathcal{M}, u\models^i \chi)\\
	\mathcal{M}, w&\models^i \psi \diamondto \chi \Leftrightarrow (\exists u \in W)(R_{\|\psi\|^i_\mathcal{M}}(w, u)\text{ and }\mathcal{M}, u\models^i \chi)	
\end{align*}
where we assume, for any given $\phi \in \mathcal{L}$, that $\|\phi\|^i_\mathcal{M}$ stands for the set $\{w \in W\mid \mathcal{M}, w\models^i\phi\}$.

We now define that $\mathsf{IntCK}^{e+} := \mathsf{L}(\mathcal{L}^{e+}_{(\boxto,\diamondto)}, CInt^e, \models^i)$. Of course the non-positive variant of the same logic is also possible; namely we can define $\mathsf{IntCK} := \mathsf{L}(\mathcal{L}_{(\boxto,\diamondto)}, CInt, \models^i)$. The paper \cite{olkhovikov} looks into $\mathsf{IntCK}$ in more detail.
 
It is easy to show that the $\{\boxto, \diamondto\}$-free fragment of $\mathsf{IntCK}^{e+}$ is exactly $\mathsf{IL}^{e+}$. It is therefore natural to expect, again, that the relation between $\mathsf{N4CK}$ and $\mathsf{IntCK}^{e+}$ resembles the relation between their propositional bases $\mathsf{N4}$ and $\mathsf{IL}^{e+}$ and that, therefore, some variant of Proposition \ref{P:prop-embedding} can be proven for the two conditional logics.

This is indeed the case. One option is to extend $E$ to the following mapping $E^\pm: (\mathcal{L}_{\boxto})\to \mathcal{L}^{e+}_{(\boxto,\diamondto)}$; its definition includes every clause from the definition of $E$ plus the following clauses for $\boxto$:
\begin{align*}
	E^\pm(\phi\boxto \psi)&:= E^\pm(\phi)\boxto (E^\pm(\sim\phi) \boxto E^\pm(\psi))\\
	E^\pm(\sim(\phi\boxto \psi))&:= E^\pm(\phi)\diamondto (E^\pm(\sim\phi) \diamondto E^\pm(\sim\psi))
\end{align*}
We begin by proving two technical lemmas which are in an obvious correspondence with Claims 1 and 2 made in the proof of Proposition \ref{P:prop-embedding}:

\begin{lemma}\label{L:claim1}
	Assume that $\mathcal{M} = (W, \leq, R, V^+, V^-) \in CNel$. Then let $\mathcal{M}^i = (W^i, \leq^i, R^i, V^i)$ be defined as follows:
	\begin{itemize}
		\item $W^i:= W \cup R = W \cup \{(w, (X,Y), v)\mid (w, (X,Y), v) \in R\}$.
		
		\item $\leq^i:= \leq \cup \{((w, (X,Y), v), (w', (X,Y), v'))\mid w \leq w',\,v\leq v'\}$.
		
		\item $R^i:= \{(w, X', (w, (X,Y), v)), ((w, (X,Y), v), Y', v)\mid w,v \in W,\,(w, (X,Y), v) \in R,\,X'\cap W = X,\,Y'\cap W = Y\}$.
		
		\item $V(p_i): = V^+(p_i)$ and $V(q_i) := V^-(p_i)$ for every $i \in \omega$. 
	\end{itemize} 
Then the following statements are true:
\begin{enumerate}
	\item $\mathcal{M}^i \in CInt^{e}$.
	
	\item For every $w \in W$ and every $\phi \in \mathcal{L}_{\boxto}$, we have $\mathcal{M}, w \models^+ \phi$ iff $\mathcal{M}^i, w \models^i E^\pm(\phi)$.
\end{enumerate}
\end{lemma}
\begin{proof}
	(Part 1) The only non-trivial part is the satisfaction of conditions \eqref{Cond:1i} and \eqref{Cond:2i} by $\mathcal{M}^i$. We reason as follows:
	
	\textit{Condition} \eqref{Cond:1i}. Let $\alpha, \beta, \gamma \in W^i$ and let $\Xi\subseteq W^i$ be such that $\alpha\geq^i\beta\mathrel{R^i_\Xi}\gamma$. Then the following cases are possible:
	
	\textit{Case 1}. For some $w, v \in W$ and some $X, Y \subseteq W$, we have $\beta = w$, $\gamma = (w, (X,Y), v) \in R$, and $X = \Xi \cap W$. But then, since $w = \beta\leq^i\alpha$, we must also have $\alpha = u$ for some $u \in W$ and $w \leq u$. By condition \eqref{Cond:1} for $\mathcal{M}$, choose a $u' \in W$ such that both $u'\geq v$ and $(u, (X, Y), u') \in R \subseteq W^i$. By definition of $\mathcal{M}^i$, we have then both $(w, (X,Y), v)\leq^i(u, (X, Y), u')$ and $R^i_\Xi(u, (u, (X, Y), u'))$.
	
	\textit{Case 2}. For some $w, v \in W$ and some $X, Y \subseteq W$, we have $\beta = (w, (X,Y), v) \in R$, $\gamma = v$, and $Y = \Xi \cap W$. But then, since $(w, (X,Y), v) = \beta\leq^i\alpha$, we must also have $\alpha = (w', (X,Y), v')\in R$ for some $w',v' \in W$ such that both $w \leq w'$ and $v \leq v'$. By definition of $\mathcal{M}^i$, we will have then both $v\leq^iv'$ and $R^i_\Xi((w', (X, Y), v'), v')$.
	
	\textit{Condition} \eqref{Cond:2i}. Let $\alpha, \beta, \gamma \in W^i$ and let $\Xi\subseteq W^i$ be such that $\alpha\mathrel{R^i_\Xi}\beta\leq^i\gamma$. Then the following cases are possible:
	
	\textit{Case 1}. For some $w, v \in W$ and some $X, Y \subseteq W$, we have $\alpha = w$, $\beta = (w, (X,Y), v) \in R$, and $X = \Xi \cap W$. But then, since $(w, (X,Y), v) = \beta\leq^i\gamma$, we must also have $\gamma = (w', (X,Y), v')\in R$ for some $w',v' \in W$ such that both $w \leq w'$ and $v \leq v'$. By definition of $\mathcal{M}^i$, we will have then both $w\leq^iw'$ and $R^i_\Xi(w',(w', (X, Y), v'))$.
	
	\textit{Case 2}. For some $w, v \in W$ and some $X, Y \subseteq W$, we have $\alpha = (w, (X,Y), v) \in R$, $\beta = v$, and $Y = \Xi \cap W$. But then, since $v = \beta\leq^i\gamma$, we must also have $\gamma = u$ for some $u \in W$ and $w \leq u$. By condition \eqref{Cond:2} for $\mathcal{M}$, choose a $u' \in W$ such that both $u'\geq w$ and $(u', (X, Y), u) \in R \subseteq W^i$. By definition of $\mathcal{M}^i$, we have then both $(w, (X,Y), v)\leq^i(u', (X, Y), u)$ and $R^i_\Xi((u', (X, Y), u), u)$.
	
	(Part 2) We proceed by induction on the construction of $\phi \in \mathcal{L}_{\boxto}$. The basis and the induction step for $\sim$, $\wedge$, and $\vee$ are straightforward. We consider the remaining cases.
	
	\textit{Case 1}. $\phi = (\psi \to \chi)$. If now $w \in W$ then $\mathcal{M}, w \models^+ \phi$ iff $(\forall v \geq w)(\mathcal{M}, v \models^+ \psi\text{ implies }\mathcal{M}, v \models^+ \chi)$, iff, by definition of $\mathcal{M}^i$, $(\forall v \geq^i w)(\mathcal{M}, v \models^+ \psi\text{ implies }\mathcal{M}, v \models^+ \chi)$, iff, by IH, $(\forall v \geq^i w)(\mathcal{M}^i, v \models^i E^\pm(\psi)\text{ implies }\mathcal{M}^i, v \models^i E^\pm(\chi))$; the latter is clearly equivalent to $\mathcal{M}^i, w \models^i E^\pm(\phi)$.
	
	\textit{Case 2}. $\phi = \sim(\psi \to \chi)$. If now $w \in W$ then $\mathcal{M}, w \models^+ \phi$ iff $(\mathcal{M}, v \models^+ \psi\text{ and }\mathcal{M}, v \models^+ \sim\chi)$, iff, by IH, $(\mathcal{M}^i, w \models^i E^\pm(\psi)\text{ and }\mathcal{M}^i, w \models^i E^\pm(\sim\chi))$ iff $\mathcal{M}^i, w \models^i E^\pm(\phi)$.
	
	\textit{Case 3}. $\phi = (\psi \boxto \chi)$. Let $w \in W$ be arbitrary. ($\Leftarrow$) If $\mathcal{M}, w \not\models^+ \phi$, then there must be some $v, u \in W$ such that $w\leq v\mathrel{R_{\|\psi\|_\mathcal{M}}}u$ and $\mathcal{M}, u \not\models^+ \chi$. Thus, in particular, $(v, \|\psi\|_\mathcal{M}, u)\in R$. By IH, we know that $\mathcal{M}^i, u \not\models^i E^\pm(\chi)$, and, moreover, that both $\|E^\pm(\psi)\|^i_{\mathcal{M}^i} \cap W = \|\psi\|^+_\mathcal{M}$ and $\|E^\pm(\sim\psi)\|^i_{\mathcal{M}^i} \cap W = \|\sim\psi\|^+_\mathcal{M} = \|\psi\|^-_\mathcal{M}$. By definition of $\mathcal{M}^i$, we know that $w, v, u, (v, \|\psi\|_\mathcal{M}, u) \in W^i$, that $w\leq^i v$, that $R^i_{\|E^\pm(\psi)\|^i_{\mathcal{M}^i}}(v, (v, \|\psi\|_\mathcal{M}, u))$, that $(v, \|\psi\|_\mathcal{M}, u)\leq^i (v, \|\psi\|_\mathcal{M}, u)$, and that $R^i_{\|E^\pm(\sim\psi)\|^i_{\mathcal{M}^i}}((v, \|\psi\|_\mathcal{M}, u), u)$. These facts allow us to conclude that we have both $
	\mathcal{M}^i, (v,\|\psi\|_\mathcal{M}, u)\not\models^i E^\pm(\sim\psi)\boxto E^\pm(\chi)$ and $
	\mathcal{M}^i, w\not\models^i (E^\pm(\psi)\boxto (E^\pm(\sim\psi)\boxto E^\pm(\chi))) = E^\pm(\phi)$. 	
	
	($\Rightarrow$). If $\mathcal{M}^i, w\not\models^i E^\pm(\phi) = (E^\pm(\psi)\boxto (E^\pm(\sim\psi)\boxto E^\pm(\chi)))$, then there must be $\alpha, \beta, \gamma, \delta \in W^i$ such that we have $w\leq^i\alpha\mathrel{R^i_{\|E^\pm(\psi)\|^i_{\mathcal{M}^i}}}\beta\leq^i\gamma{R^i_{\|E^\pm(\sim\psi)\|^i_{\mathcal{M}^i}}}\delta$ and also $\mathcal{M}^i,\delta\not\models^i E^\pm(\chi)$. Since $W\ni w\leq^i\alpha$, we know, by definition of $\mathcal{M}^i$, that $\alpha = v$ for some $W \ni v \geq w$. Next, IH implies that $\|E^\pm(\psi)\|^i_{\mathcal{M}^i} \cap W = \|\psi\|^+_\mathcal{M}$, therefore, by definition of $R^i$, we have $\beta = (v, (\|\psi\|^+_\mathcal{M}, Y), u) \in R$ for some $u \in W$ and some $Y \subseteq W$. The definition of $\leq^i$ now implies that we must have $\gamma = (v', (\|\psi\|^+_\mathcal{M}, Y), u') \in R$ for some $v', u' \in W$ such that both $v \leq v'$ and $u\leq u'$. Since, therefore, we must have $\gamma = (v', (\|\psi\|^+_\mathcal{M}, Y), u')R^i_{\|E^\pm(\sim\psi)\|^i_{\mathcal{M}^i}}\delta$, the definition of $R^i$, together with IH implies now that $Y = \|E^\pm(\sim\psi)\|^i_{\mathcal{M}^i} \cap W = \|\sim\psi\|^+_\mathcal{M} = \|\psi\|^-_\mathcal{M}$ and that $\delta = u'$. We must have, therefore, all of the following:
	$$
	\alpha = v\geq w,\,\beta = (v, \|\psi\|_\mathcal{M}, u) \in R,\,v \leq v',\,u\leq u',\,\gamma = (v', \|\psi\|_\mathcal{M}, u')\in R,\,\delta = u'.
	$$
By transitivity of $\leq$, we also get that $w \leq v'$. Looking at these facts from the standpoint of $\mathcal{M}$, we obtain the following diagram:
\begin{diagram}
	& & v'&  \rTo_{R_{\|\psi\|_\mathcal{M}}}& u' = \delta &\\
	 &\ruDotsto_\leq  & \uTo_\leq &             & \uTo_\leq &\\
	w & \rTo_{\leq} & v = \alpha &\rTo_{R_{\|\psi\|_\mathcal{M}}}  & u & 
\end{diagram}
Moreover, the fact that $\mathcal{M}^i,\delta = u'\not\models^i E^\pm(\chi)$ implies, by IH, that $\mathcal{M},u'\not\models^+ \chi$ so that $\mathcal{M}, w \not\models^+ (\psi\boxto \chi) = \phi$ clearly follows.

\textit{Case 4}. $\phi = \sim(\psi \boxto \chi)$. Let $w \in W$ be arbitrary. ($\Rightarrow$) If $\mathcal{M}, w \models^+ \phi$, then  $\mathcal{M}, w \models^- \psi \boxto \chi$, and there must be some $v \in W$ such that $w\mathrel{R_{\|\psi\|_\mathcal{M}}}v$ and $\mathcal{M}, v \models^- \chi$, or, equivalently, $\mathcal{M}, v \models^+ \sim\chi$. Thus, in particular, $(w, ( \|\psi\|^+_\mathcal{M},  \|\psi\|^-_\mathcal{M}), v)\in R$. By IH, we know that $\mathcal{M}^i, v \models^i E^\pm(\sim\chi)$, and, moreover, that both $\|E^\pm(\psi)\|^i_{\mathcal{M}^i} \cap W = \|\psi\|^+_\mathcal{M}$ and $\|E^\pm(\sim\psi)\|^i_{\mathcal{M}^i} \cap W = \|\sim\psi\|^+_\mathcal{M} = \|\psi\|^-_\mathcal{M}$. By definition of $\mathcal{M}^i$, we know that $w, v, (w, ( \|\psi\|^+_\mathcal{M},  \|\psi\|^-_\mathcal{M}), v) \in W^i$, that $R^i_{\|E^\pm(\psi)\|^i_{\mathcal{M}^i}}(w, (w, ( \|\psi\|^+_\mathcal{M},  \|\psi\|^-_\mathcal{M}), v))$, and that $R^i_{\|E^\pm(\sim\psi)\|^i_{\mathcal{M}^i}}((w, ( \|\psi\|^+_\mathcal{M},  \|\psi\|^-_\mathcal{M}), v), v)$. These facts allow us to conclude that, first, we have $
\mathcal{M}^i, (w,(\|\psi\|^+_\mathcal{M},  \|\psi\|^-_\mathcal{M}), v)\models^i E^\pm(\sim\psi)\diamondto E^\pm(\sim\chi)$,
and, second, that $
\mathcal{M}^i, w\models^i (E^\pm(\psi)\diamondto (E^\pm(\sim\psi)\diamondto E^\pm(\sim\chi))) = E^\pm(\phi)$.

($\Leftarrow$). If $\mathcal{M}^i, w\models^i E^\pm(\phi) = (E^\pm(\psi)\diamondto (E^\pm(\sim\psi)\diamondto E^\pm(\sim\chi)))$, then there must be $\alpha, \beta \in W^i$ such that we have $w\mathrel{R^i_{\|E^\pm(\psi)\|^i_{\mathcal{M}^i}}}\alpha\mathrel{R^i_{\|E^\pm(\sim\psi)\|^i_{\mathcal{M}^i}}}\beta$ and also $\mathcal{M}^i,\beta\models^i E^\pm(\sim\chi)$. Now, IH implies that $\|E^\pm(\psi)\|^i_{\mathcal{M}^i} \cap W = \|\psi\|^+_\mathcal{M}$, therefore, by definition of $R^i$, we must have $\alpha = (w, (\|\psi\|^+_\mathcal{M}, Y), v) \in R$ for some $v \in W$ and some $Y \subseteq W$. In other words, $\alpha = (w, (\|\psi\|^+_\mathcal{M}, Y), v)R^i_{\|E^\pm(\sim\psi)\|^i_{\mathcal{M}^i}}\beta$, whence, by the definition of $R^i$ together with IH, we know that $Y = \|E^\pm(\sim\psi)\|^i_{\mathcal{M}^i} \cap W = \|\sim\psi\|^+_\mathcal{M} = \|\psi\|^-_\mathcal{M}$ and that $\beta = v$. Summing up, we must have
	$\alpha = (w, \|\psi\|_\mathcal{M}, v) \in R$ and $\beta = v$,
which implies, in particular, that $w\mathrel{R_{\|\psi\|_\mathcal{M}}}v$ Moreover, the fact that $\mathcal{M}^i,\beta = v\models^i E^\pm(\sim\chi)$ implies, by IH, that $\mathcal{M},v\models^+ \sim\chi$, or, equivalently, that $\mathcal{M},v\models^- \chi$ so that we clearly have $\mathcal{M}, w \models^- (\psi\boxto \chi)$ and thus also $\mathcal{M}, w \models^+ \sim(\psi\boxto \chi) = \phi$, as desired.
 \end{proof}
\begin{lemma}\label{L:claim2}
	Assume that $\mathcal{M} = (W, \leq, R, V)\in CInt^e$. Then let the model $\mathcal{M}^{n4} = (W, \leq, R^{n4}, V^+, V^-)$ be defined as follows:
	\begin{itemize}
		\item $R^{n4}:= \{(w, (X,Y), v)\mid \exists u((w, X, u)\in R\,\&\,(u, Y, v)\in R)\}$.
		\item $V^+(p_i): = V(p_i)$ and $V^-(p_i) := V(q_i)$ for every $i \in \omega$. 
	\end{itemize} 
	Then the following statements are true:
	\begin{enumerate}
		\item $\mathcal{M}^{n4}\in CNel$.
		
		\item For every $w \in W$ and every $\phi \in \mathcal{L}_{\boxto}$, we have $\mathcal{M}^{n4}, w \models^+ \phi$ iff $\mathcal{M}, w \models^i E^\pm(\phi)$.
	\end{enumerate}
\end{lemma}
\begin{proof}
	(Part 1). Again, the only non-trivial part is the satisfaction of conditions \eqref{Cond:1} and \eqref{Cond:2} by $\mathcal{M}^{n4}$. We reason as follows:
	
	\textit{Condition} \eqref{Cond:1}. Let $w, v, u \in W$ and $X, Y\subseteq W$ be such that $w\geq v\mathrel{R^{n4}_{(X,Y)}}u$. Then, for some $u' \in W$ we must have both $v\mathrel{R_{X}}u'$ and $u'\mathrel{R_{Y}}u$. By condition \eqref{Cond:1i} for $\mathcal{M}$, there must be a $w_0 \in W$ such that $w\mathrel{R_{X}}w_0\geq u'$. Applying the same condition one more time, we see that there must also exist a $w_1 \in W$ such that $w_0\mathrel{R_{Y}}w_1\geq u$. But then the definition of $R^{n4}$ implies that we must have $w\mathrel{R^{n4}_{(X,Y)}}w_1\geq u$. 
	
	\textit{Condition} \eqref{Cond:2}. Let $w, v, u \in W$ and $X, Y\subseteq W$ be such that $w\mathrel{R^{n4}_{(X,Y)}}v\leq u$. Then, for some $v' \in W$ we must have both $w\mathrel{R_{X}}v'$ and $v'\mathrel{R_{Y}}v$. By condition \eqref{Cond:2i} for $\mathcal{M}$, there must be a $w_1 \in W$ such that $v'\leq w_1\mathrel{R_{Y}}u$. Applying the same condition one more time, we see that there must also exist a $w_0 \in W$ such that $w\leq w_0\mathrel{R_{X}}w_1$. But then the definition of $R^{n4}$ implies that we must have $w\leq w_0\mathrel{R^{n4}_{(X,Y)}}u$. 
	
	(Part 2). We proceed by induction on the construction of $\phi \in \mathcal{L}_{\boxto}$. The basis and the induction step for $\sim$, $\wedge$, $\vee$, and $\to$ are straightforward. We consider the remaining cases.

	\textit{Case 1}. $\phi = (\psi \boxto \chi)$. Let $w \in W$ be arbitrary. ($\Leftarrow$) If $\mathcal{M}, w \not\models^+ \phi$, then there must be some $v, u \in W$ such that $w\leq v\mathrel{R^{n4}_{\|\psi\|_{\mathcal{M}^{n4}}}}u$ and $\mathcal{M}^{n4}, u \not\models^+ \chi$. By definition of $R^{n4}$, there must be a $u' \in W$ such that both $v\mathrel{R_{\|\psi\|^+_{\mathcal{M}^{n4}}}}u'$ and $u'\mathrel{R_{\|\psi\|^-_{\mathcal{M}^{n4}}}}u$. By IH, we know that $\mathcal{M}, u \not\models^i E^\pm(\chi)$, and, moreover, that both $\|E^\pm(\psi)\|^i_{\mathcal{M}}  = \|\psi\|^+_{\mathcal{M}^{n4}}$ and $\|E^\pm(\sim\psi)\|^i_{\mathcal{M}} = \|\sim\psi\|^+_{\mathcal{M}^{n4}} = \|\psi\|^-_{\mathcal{M}^{n4}}$. Finally, note that we have $u' \leq u'$ by reflexivity. These facts allow us to conclude that both $
	\mathcal{M}, u'\not\models^i E^\pm(\sim\psi)\boxto E^\pm(\chi)$ and $
	\mathcal{M}, w\not\models^i (E^\pm(\psi)\boxto (E^\pm(\sim\psi)\boxto E^\pm(\chi))) = E^\pm(\phi)$.
	
	($\Rightarrow$). If $\mathcal{M}, w\not\models^i E^\pm(\phi) = (E^\pm(\psi)\boxto (E^\pm(\sim\psi)\boxto E^\pm(\chi)))$, then there must be $v, u \in W$ such that we have both $w\leq v\mathrel{R_{\|E^\pm(\psi)\|^i_{\mathcal{M}}}}u$ and $\mathcal{M}, u\not\models^i E^\pm(\sim\psi)\boxto E^\pm(\chi)$; the latter further implies that there must be some $v', u' \in W$ such that both $u\leq v'\mathrel{R_{\|E^\pm(\sim\psi)\|^i_{\mathcal{M}}}}u'$ and $\mathcal{M}, u'\not\models^i E^\pm(\chi)$. Now, IH implies that $\|E^\pm(\psi)\|^i_{\mathcal{M}}  = \|\psi\|^+_{\mathcal{M}^{n4}}$, that $\|E^\pm(\sim\psi)\|^i_{\mathcal{M}} = \|\sim\psi\|^+_{\mathcal{M}^{n4}} = \|\psi\|^-_{\mathcal{M}^{n4}}$, and that also $\mathcal{M}^{n4}, u'\not\models^+\chi$. It follows that $w\leq v\mathrel{R_{\|\psi\|^+_{\mathcal{M}^{n4}}}}u$ and $u\leq v'\mathrel{R_{\|\psi\|^-_{\mathcal{M}^{n4}}}}u'$.
Next, by condition \eqref{Cond:2i} we choose a $w' \in W$ such that $v\leq w'\mathrel{R_{\|\psi\|^+_{\mathcal{M}^{n4}}}}v'$. The whole situation is then represented in the following diagram:
\begin{diagram}
	 w' & \rDotsto_{R_{\|\psi\|^+_{\mathcal{M}^{n4}}}}                                        &  v'       & \rTo_{R_{\|\psi\|^-_{\mathcal{M}^{n4}}}} & u'&\\
	\uDotsto_\leq  &                                         & \uTo_\leq &\\
 v&  \rTo_{R_{\|\psi\|^+_{\mathcal{M}^{n4}}}}& u        &\\
	\uTo_\leq  &  &             & &\\
	w &  & & & & 
\end{diagram}
By transitivity of $\leq$ and the definition of $R^{n4}$, we get that $w\leq w'\mathrel{R^{n4}_{\|\psi\|_{\mathcal{M}^{n4}}}}u'$, which, together with $\mathcal{M}^{n4}, u'\not\models^+\chi$, implies that $\mathcal{M}, w \not\models^+ (\psi\boxto\chi) = \phi$.
		
	\textit{Case 2}. $\phi = \sim(\psi \boxto \chi)$. Let $w \in W$ be arbitrary. Then $\mathcal{M}^{n4}, w \models^+ \phi$ iff  $\mathcal{M}, w \models^- \psi \boxto \chi$ iff for some $v \in W$ we have $w\mathrel{R^{n4}_{\|\psi\|_{\mathcal{M}^{n4}}}}v$ and $\mathcal{M}^{n4}, v \models^- \chi$, iff for some $v \in W$ we have $w\mathrel{R^{n4}_{\|\psi\|_{\mathcal{M}^{n4}}}}v$ and $\mathcal{M}^{n4}, v \models^+ \sim\chi$, iff, for some $v, u \in W$ we have $w\mathrel{R_{\|\psi\|^+_{\mathcal{M}^{n4}}}}u$, $u\mathrel{R_{\|\psi\|^-_{\mathcal{M}^{n4}}}}u$ and $\mathcal{M}^{n4}, v \models^+ \sim\chi$. By IH, the latter holds iff for some $u, v \in W$ we have $w\mathrel{R_{\|E^\pm(\psi)\|^i_{\mathcal{M}}}}u$, $u\mathrel{R_{\|E^\pm(\sim\psi)\|^i_{\mathcal{M}}}}v$ and $\mathcal{M}, v \models^iE^\pm(\sim\chi)$. But the latter is equivalent to $
	\mathcal{M}, w\models^i (E^\pm(\psi)\diamondto (E^\pm(\sim\psi)\diamondto E^\pm(\sim\chi))) = E^\pm(\phi)$.
\end{proof}
We are now ready to state and prove the faithfulness of $E^\pm$:
\begin{proposition}\label{R:Em}
	Let $\Gamma, \Delta \subseteq \mathcal{L}_{\boxto}$. Then $(\Gamma, \Delta)\in \mathsf{N4CK}$ iff $(E^\pm(\Gamma), E^\pm(\Delta))\in \mathsf{IntCK}^{e+}$.
\end{proposition}
\begin{proof}
	We argue as in the proof of Proposition \ref{P:prop-embedding}. More precisely, if $\mathcal{M}, w \models (\Gamma, \Delta)$, then, by Lemma \ref{L:claim1}, $\mathcal{M}^i, w \models_{\mathsf{IntCK}^{e+}}(E^\pm(\Gamma), E^\pm(\Delta))$. Conversely, if $\mathcal{M}, w\models_{\mathsf{IntCK}^{e+}}(E^\pm(\Gamma), E^\pm(\Delta))$, then, by Lemma \ref{L:claim2}, $\mathcal{M}^{n4}, w \models (\Gamma, \Delta)$.  
\end{proof}
\begin{remark}\label{R:3}
1. There are alternative ways to define the embedding of $\mathsf{N4CK}$ into $\mathsf{IntCK}^{e+}$. For example, the mapping $E^\mp$ obtained by extending the definition of $E$ with the following clauses:
\begin{align*}
	E^\mp(\phi\boxto \psi)&:= E^\mp(\sim\phi)\boxto (E^\mp(\phi) \boxto E^\mp(\psi))\\
	E^\mp(\sim(\phi\boxto \psi))&:= E^\mp(\sim\phi)\diamondto (E^\mp(\phi) \diamondto E^\mp(\sim\psi))
\end{align*}
is also correct.

2. However, some obvious simplifications of $E^\pm$ fail to correctly embed $\mathsf{N4CK}$ into $\mathsf{IntCK}^{e+}$. For example, consider the mapping $E^+$ obtained from $E$ by extending its definition with the clauses
\begin{align*}
	E^+(\phi\boxto \psi)&:= E^+(\phi) \boxto E^+(\psi)\\
	E^+(\sim(\phi\boxto \psi))&:= E^+(\phi) \diamondto E^+(\sim\psi)
\end{align*}
$E^+$ fails to correctly embed $\mathsf{N4CK}$ into $\mathsf{IntCK}^{e+}$ since we have $\{(p_1 \wedge \sim p_2)\boxto p_3\} \not\models_{\mathsf{N4CK}} \{(\sim(p_1 \to p_2))\boxto p_3\}$. Indeed, just consider any $\mathcal{M} = (W, \leq, R, V^+, V^-)\in CNel$ such that $\|\sim(p_1\to p_2)\|^-_\mathcal{M}\neq \|p_1\wedge\sim p_2\|^-_\mathcal{M}$, and $V^+(p_3) = R_{\|p_1 \wedge \sim p_2\|_\mathcal{M}} = \emptyset \neq R_{\|\sim(p_1 \to p_2)\|_\mathcal{M}}$. On the other hand, we have:
$$
E^+((p_1 \wedge \sim p_2)\boxto p_3) = (p_1\wedge q_2)\boxto p_3 = E^+((\sim(p_1 \to p_2))\boxto p_3),
$$
therefore, we must also have $\{E^+((p_1 \wedge \sim p_2)\boxto p_3)\} \models_{\mathsf{IntCK}^+} \{E^+((\sim(p_1 \to p_2))\boxto p_3)\}$.

Similarly, the mapping $E^-$ obtained by extending the definition of $E$ with the following clauses:
\begin{align*}
	E^-(\phi\boxto \psi)&:= E^-(\sim\phi) \boxto E^-(\psi)\\
	E^-(\sim(\phi\boxto \psi))&:= E^-(\sim\phi) \diamondto E^-(\sim\psi)
\end{align*}
does not give a faithful embedding, since we have  $\{(\sim(p_1 \wedge \sim p_2))\boxto p_3\} \not\models_{\mathsf{N4CK}} \{(p_1 \to p_2)\boxto p_3\}$; however, due to the equalities 
$$
E^-((\sim(p_1 \wedge \sim p_2))\boxto p_3) = (p_1\wedge q_2)\boxto p_3 = E^-((p_1 \to p_2)\boxto p_3),
$$
we must also have $\{E^-((\sim(p_1 \wedge \sim p_2))\boxto p_3)\} \models_{\mathsf{IntCK}^{e+}} \{E^-((p_1 \to p_2)\boxto p_3)\}$.

3. On the other hand, it is easy to see that $E^+$ faithfully embeds $\mathsf{N4CK}'$ into $\mathsf{IntCK}^{e+}$.

4. Moreover, one can show that, for all  $\Gamma, \Delta \subseteq \mathcal{L}_{\boxto}$, $\Gamma\models_{\mathsf{N4CK}}\Delta$ \textit{implies} both $E^+(\Gamma)\models_{\mathsf{IntCK}^{e+}} E^+(\Delta)$	and $E^-(\Gamma)\models_{\mathsf{IntCK}^{e+}} E^-(\Delta)$; thus, no counterexamples can be given to the other direction in the faithfulness claim for $E^+$ and $E^-$. Indeed, the construction given in Lemma \ref{L:claim2} can be straightforwardly adapted to the respective definitions of $E^+$ and $E^-$; for example, in the case of $E^+$ one needs to set $R^{n4}:= \{(w, (X,Y), v)\mid ((w, X, v)\in R\}$.
\end{remark}

\subsection{Modal logics}\label{sub:modal}
Conditional logics often have modal companions that are faithfully embeddable into them by a simple and natural translation. Since $\mathsf{CK}$ and $\mathsf{IntCK}$ can be viewed as basic classical and intuitionistic conditional logic, respectively, their natural companions are provided by the minimal normal modal logic $\mathsf{K}$ and the basic intuitionistic modal logic $\mathsf{IK}$ introduced by G. Fischer-Servi in \cite{fischer-servi}. We will argue in the present subsection that this series can be continued with $\mathsf{N4CK}$ and the modal logic $\mathsf{FSK}^d$ introduced in \cite{odintsovwansing}.

But first let us briefly discuss the modal logics in question. Whereas $\mathsf{K}$ (understood here over $\mathcal{L}_\Box$) is relatively well-known, both $\mathsf{IK}$ and $\mathsf{FSK}^d$ require an introduction. We start by defining the Nelsonian modal models:
\begin{definition}\label{D:modal-model}
	A Nelsonian modal model is a structure of the form $\mathcal{M} = (W, \leq, R, V^+, V^-)$, where $(W, \leq, V^+, V^-)\in Nel$ and $R \subseteq W \times W$ satisfies the following conditions:
	\begin{align}
		\leq^{-1}\circ R &\subseteq R\circ\leq^{-1}\label{Cond:1m}\tag{c1-m}\\
		R\circ\leq &\subseteq \leq\circ R\label{Cond:2m}\tag{c2-m}
	\end{align}	
The model is extended, if $V^+$ and $V^-$ are defined on $Var \cup Var'$ instead of $Var$.
\end{definition}
The class of all (extended) Nelsonian modal models will be denoted by $MNel$ (resp. $MNel^e$); for any $\mathcal{M} = (W, \leq, R, V^+, V^-) \in MNel$ (resp. $MNel^e$), the structure $(W, \leq, R, V^+)$ is called an (extended) intuitionistic modal model. The class of all (extended) intuitionistic modal models will be denoted by $MInt$ (resp. $MInt^e$). 

It remains to define the satisfaction relations for the two modal logics. For $\mathsf{FSK}^d$, we have the pair of satisfaction relations denoted by $\models^+_m$ and $\models^-_m$, respectively. Their definition is given by induction on the construction of a formula in $\mathcal{L}_\Box$ and extends the inductive definitions of $\models^+_{\mathsf{N4}}$ and $\models^-_{\mathsf{N4}}$ by adding the following clauses for $\Box$:
\begin{align*}
	\mathcal{M}, w&\models_m^+ \Box\psi \text{ iff } (\forall v \geq w)(\forall u \in W)(R(v, u) \text{ implies }\mathcal{M}, u\models^+ \psi)\\
	\mathcal{M}, w&\models_m^- \Box\psi \text{ iff } (\exists u \in W)(R(w, u)\text{ and }\mathcal{M}, u\models^- \chi)	
\end{align*}
As for $\mathsf{IK}$, we only define a single satisfaction relation $\models^i_m$, which extends the definition of $\models_{\mathsf{IL}}$ with the two additional clauses:
\begin{align*}
	\mathcal{M}, w&\models^i_m \Box\psi \text{ iff } (\forall v \geq w)(\forall u \in W)(R(v, u) \text{ implies }\mathcal{M}, u\models^i_m \psi)\\
	\mathcal{M}, w&\models^i_m \Diamond\psi \text{ iff } (\exists u \in W)(R(w, u)\text{ and }\mathcal{M}, u\models^i_m \psi)	
\end{align*}
We now define our logics\footnote{The authors of \cite{odintsovwansing} define $\mathsf{FSK}^d$ over $\mathcal{L}_{(\Box,\Diamond)}$ instead and also use two accessibility relations, $R_\Box$ and $R_\Diamond$ instead of just $R$. The reason is that they are interested in a general semantic framework covering both $\mathsf{FSK}^d$ and a number of weaker $\mathsf{N4}$-based modal logics. The simpler definitions of the present paper are easily seen to yield the same semantics for $\mathsf{FSK}^d$ as the more complicated setting chosen in  \cite{odintsovwansing}.} as follows:
$$
	\mathsf{FSK}^d:= \mathsf{L}(\mathcal{L}_\Box, MNel, \models_m^+);\,\mathsf{IK}:= \mathsf{L}(\mathcal{L}_{(\Box,\Diamond)}, MInt, \models_m^i).
$$

Both $\mathsf{K}$ and $\mathsf{IK}$ (defined over $\mathcal{L}_\Box$ and $\mathcal{L}_{(\Box,\Diamond)}$, respectively) are embedded into their corresponding conditional logics, that is to say, into $\mathsf{CK}$ and $\mathsf{IntCK}$, respectively, by what is essentially one and the same translation mapping. To make the matters more precise, let $\phi\in \mathcal{L}_{(\boxto,\diamondto)}$. The mapping $Tr^i_\phi:\mathcal{L}_{(\Box,\Diamond)}\to\mathcal{L}_{(\boxto,\diamondto)}$ is defined by the following induction on the construction of $\psi\in \mathcal{L}_{(\Box,\Diamond)}$:
\begin{align*}
	Tr^i_\phi(p)&:= p&&p\in Var\\
	Tr^i_\phi(\sim\psi)&:=\sim Tr^i_\phi(\psi)\\
	Tr^i_\phi(\psi\ast\chi)&:= Tr^i_\phi(\psi)\ast Tr^i_\phi(\chi)&&\ast\in \{\wedge, \vee, \to\}\\
	Tr^i_\phi(\Box\psi)&:= \phi\boxto Tr^i_\phi(\psi)&&Tr^i_\phi(\Diamond\psi):= \phi\diamondto Tr^i_\phi(\psi)
\end{align*}
For $\phi\in \mathcal{L}_{\boxto}$, we also set\footnote{Another option would be just to leave $Tr^i$ untouched while reading $\Diamond\chi$ as the abbreviation for $\sim\Box\sim\chi$ and $\chi\diamondto\theta$ as the abbreviation for $\sim(\chi\boxto\sim\theta)$, with the same result.} $Tr_\phi:= (Tr^i_\phi\upharpoonright\mathcal{L}_\Box):\mathcal{L}_\Box\to\mathcal{L}_{\boxto}$.

The following results have been established earlier:
\begin{proposition}\label{P:embeddings}
	Let $\phi\in \mathcal{L}_{(\boxto,\diamondto)}$ and let $\Gamma, \Delta\subseteq \mathcal{L}_{(\Box,\Diamond)}$. Then the following statements hold:
	\begin{enumerate}
		\item $\Gamma\models_{\mathsf{IK}}\Delta$ iff $Tr^i_\phi(\Gamma)\models_{\mathsf{IntCK}}Tr^i_\phi(\Delta)$.
		
		\item In case  $\phi\in \mathcal{L}_{\boxto}$ and $\Gamma, \Delta\subseteq \mathcal{L}_{\Box}$, then also $\Gamma\models_{\mathsf{K}}\Delta$ iff $Tr_\phi(\Gamma)\models_{\mathsf{CK}}Tr_\phi(\Delta)$.
	\end{enumerate}
\end{proposition}
Part 1 for the provable formulas relative to $Tr_\top$ can be found in \cite[Proposition 4]{olkhovikov}; Part 2 for the provable formulas relative to $Tr_\top$  was claimed, e.g.. in \cite[Theorem 4]{weiss-thesis}. The idea to use $Tr_\top$ to relate modal and conditional logics seems to be due to \cite{lowe}. However, it is easy to see that the same argument goes through for the provable formulas relative to $Tr_\phi$ for any $\phi$ in the target language of the embedding. The general version of Parts 1 and 2 can then be derived as in the proof of Proposition \ref{P:embed} below.

Note, however, that the translation given by $Tr_\phi$ is different from the one used in \cite{williamson}: denoting the latter translation by $T$, the crucial inductive clause there is given by $T(\Box\phi) := \sim T(\phi)\boxto T(\phi)$ and so the antecedent of the translation is not fixed, but, generally speaking, depends on the boxed formula itself.

Our goal is now to extend Proposition \ref{P:embeddings} to $\mathsf{FSK}^d$ and $\mathsf{N4CK}$. In other words, we are going to prove the following:
\begin{proposition}\label{P:embed}
	Given a $\phi\in \mathcal{L}_{\boxto}$ and $\Gamma, \Delta\subseteq \mathcal{L}_{\Box}$, we have $\Gamma\models_{\mathsf{FSK}^d}\Delta$ iff $Tr_\phi(\Gamma)\models_{\mathsf{N4CK}}Tr_\phi(\Delta)$.
\end{proposition}
We prepare the proof\footnote{A purely model-theoretic argument is also possible, but the proof we give here is more streamlined and concise.} by recalling the Hilbert-style axiomatization of $\mathsf{FSK}^d$ and proving some theorems and derived rules in this logic. First, consider the following axioms and inference rules (where $\Diamond$ is understood as the abbreviation for $\sim\Box\sim$):
\begin{align}
	&(\Box\phi\wedge\Box\psi)\to\Box(\phi\wedge\psi)\label{E:aa1}\tag{a1}\\
	&\Box(\phi\to\phi)\label{E:aa2}\tag{a2}\\
	&\Diamond(\phi\vee\psi)\to(\Diamond\phi\vee\Diamond\psi)\label{E:aa3}\tag{a3}\\
	&\Diamond(\phi\to\psi)\to(\Box\phi\to\Diamond\psi)\label{E:aa4}\tag{a4}\\
	&(\Diamond\phi\to\Box\psi)\to\Box(\phi\to\psi)\label{E:aa5}\tag{a5}\\
	&\sim\Box\phi\leftrightarrow\Diamond\sim\phi\label{E:aa6}\tag{a6}\\
	&\text{From }\phi\to \psi\text{ infer }\Box\phi\to\Box\psi\label{E:rmbox}\tag{rm$\Box$}\\
	&\text{From }\phi\to \psi\text{ infer }\Diamond\phi\to\Diamond\psi\label{E:rmdiam}\tag{rm$\Diamond$}
\end{align}
We now claim that $
\mathsf{FSK}^d = (\mathbb{N}4+(\eqref{E:aa1}-\eqref{E:aa6};\eqref{E:rmbox},\eqref{E:rmdiam}))[\mathcal{L}_\Box]$.
The following technical lemma will be needed towards our result:
\begin{lemma}\label{L:theorems-m}
	The following theorems and derived rules are provable in $\mathsf{FSK}^d$:
	\begin{align}
		&\Box\sim\phi\leftrightarrow\sim\Diamond\phi\label{E:t1}\tag{t1}\\
		&\Box(\phi\wedge\psi)\Leftrightarrow(\Box\phi\wedge\Box\psi)\label{E:t2}\tag{t2}\\
		&(\sim\Box\phi\wedge\Box\psi)\to \sim\Box(\phi\vee\sim\psi)\label{E:t3}\tag{t3}\\
		&\text{From }\phi\leftrightarrow \psi\text{ infer }\Box\phi\leftrightarrow\Box\psi\label{E:rbox}\tag{r$\Box$}\\
		&\text{From }\sim\phi\to \sim\psi\text{ infer }\sim\Box\phi\to\sim\Box\psi\label{E:rmnbox}\tag{rm$\sim\Box$}\\
		&\text{From }\sim\phi\leftrightarrow \sim\psi\text{ infer }\sim\Box\phi\leftrightarrow\sim\Box\psi\label{E:rnbox}\tag{r$\sim\Box$}
	\end{align}
\end{lemma}  
We observe that \eqref{E:t1} is just an instance of \eqref{E:a0.1} and that \eqref{E:rbox} (resp. \eqref{E:rnbox}) easily follow from \eqref{E:rmbox} (resp. \eqref{E:rmnbox}); the rest of the proof of Lemma \ref{L:theorems-m} is relegated to Appendix \ref{A:3}. Our Claim about the axiomatization of $
\mathsf{FSK}^d$ over $\mathcal{L}_\Box$ now easily follows from \cite[Theorem 4]{odintsovwansing} and \eqref{E:t1}. The following lemma gives a restricted version of Proposition \ref{P:embed} for the provable formulas:
\begin{lemma}\label{L:embed}
	For all $\phi \in \mathcal{L}_{\boxto}$ and $\psi \in \mathcal{L}_\Box$, $\psi\in\mathsf{FSK}^d$ iff $Tr_\phi(\psi)\in\mathsf{N4CK}$.
\end{lemma}
\begin{proof}
	Choose a proof $\psi_1,\ldots,\psi_n = \psi$ in $\mathsf{FSK}^d$. Consider the sequence $Tr_\phi(\psi_1),\ldots,Tr_\phi(\psi_n) = Tr_\phi(\psi)$. The application of $Tr_\phi$ leaves intact every axiom of $\mathbb{N}4$ and every application instance of \eqref{E:mp}, and maps every instance of \eqref{E:aa1}, (resp. \eqref{E:aa2}, \eqref{E:aa3}, \eqref{E:aa4}, \eqref{E:aa5}, \eqref{E:aa6}) into a consequence of \eqref{E:a1} (resp. into an instance of \eqref{E:a4}, \eqref{E:T4}, \eqref{E:T5}, \eqref{E:a3}, \eqref{E:T6}). Similarly, every application of the rule \eqref{E:rmbox} (resp. \eqref{E:rmdiam}) is mapped by $Tr_\phi$ into an application of \eqref{E:Rmbox} (resp. \eqref{E:Rmdiam}). Therefore, one can straightforwardly extend $Tr_\phi(\psi_1),\ldots,Tr_\phi(\psi_n)$ to a proof of $Tr_\phi(\psi)$ in $\mathsf{N4CK}$ by inserting the variants of deductions sketched in the proof of Lemma \ref{L:theorems}.
	
	In the other direction, let $\psi_1,\ldots,\psi_n = Tr_\phi(\psi)$ be a proof in $\mathsf{N4CK}$. Consider the mapping $\overline{Tr}:(\mathcal{L}_{\boxto})\to\mathcal{L}_\Box$ defined by induction on the construction of $\phi\in\mathcal{L}_{\boxto}$:
	\begin{align*}
		\overline{Tr}(p)&:= p&&p\in Var\\
		\overline{Tr}(\sim\psi)&:=\sim\overline{Tr}(\psi)\\
		\overline{Tr}(\psi\ast\chi)&:= \overline{Tr}(\psi)\ast \overline{Tr}(\chi)&&\ast\in \{\wedge, \vee, \to\}\\
		\overline{Tr}(\psi\boxto\chi)&:= \Box\overline{Tr}(\chi)
	\end{align*}
	The following can be easily proved by induction on the construction of $\psi\in\mathcal{L}_\Box$:
	
	\textit{Claim}. For every $\psi\in\mathcal{L}_\Box$, $\overline{Tr}(Tr_\phi(\psi)) = \psi$.
	
	Both basis and every case in the induction step are straightforward. For example, if $\psi = \Box\chi$ then $\overline{Tr}(Tr_\phi(\Box\chi)) = \overline{Tr}(\phi\boxto\chi) = \Box\chi$. Our Claim is proven.
	
	Turning back to our proof of $Tr_\phi(\psi)$ in $\mathsf{N4CK}$, we consider the sequence of $\mathcal{L}_\Box$-formulas $\overline{Tr}(\psi_1),\ldots,\overline{Tr}(\psi_n) = \overline{Tr}(Tr_\phi(\psi)) = \psi$, where the last equality holds by our Claim. We observe that the translation given by $\overline{Tr}$ leaves intact every axiom of $\mathbb{N}4$ and every application of \eqref{E:mp}; as for the other axioms and rules, $\overline{Tr}$ maps every instance of \eqref{E:a1}, (resp. \eqref{E:a2}, \eqref{E:a3}, \eqref{E:a4}) into an instance of \eqref{E:t2} (resp. \eqref{E:t3},  \eqref{E:aa5}, \eqref{E:aa2}). Similarly, every application of the rule \eqref{E:RCbox1} (resp. \eqref{E:RCbox2}) is mapped by $\overline{Tr}$ into an application of the rule \eqref{E:rbox} (resp. \eqref{E:rnbox}). Finally, the conclusion of every application of the rule \eqref{E:RAbox} is mapped by $\overline{Tr}$ into a formula of the form $\Box\psi\leftrightarrow\Box\psi\in\mathsf{N4}$. Therefore, one can extend $\overline{Tr}(\psi_1),\ldots,\overline{Tr}(\psi_n)  = \psi$ to a proof of $\psi$ in $\mathsf{FSK}^d$ by inserting the variants of deductions sketched in the proof of Lemma \ref{L:theorems-m}.
\end{proof}
We observe that the proof of Lemma \ref{L:embed} above implies the following corollary:
\begin{corollary}\label{C:embed}
	For every $\psi \in \mathsf{N4CK}$, we have $\overline{Tr}(\psi)\in\mathsf{FSK}^d$.
\end{corollary}
We are now in a position to prove the main result of this subsection:
\begin{proof}[Proof of Proposition \ref{P:embed}]
	If $\Gamma, \Delta\subseteq \mathcal{L}_{\Box}$, and  $\Gamma\models_{\mathsf{FSK}^d}\Delta$ then choose $\chi_1,\ldots,\chi_m \in \Delta$ and a deduction $\psi_1,\ldots,\psi_n = (\chi_1\vee\ldots\vee\chi_m)$ in $\mathsf{FSK}^d$ from premises in $\Gamma$. Consider the sequence $Tr_\phi(\psi_1),\ldots,Tr_\phi(\psi_n) = (Tr_\phi(\chi_1)\vee\ldots\vee Tr_\phi(\chi_m))$. By Lemma \ref{L:embed}, the application of $Tr_\phi$ brings formulas provable in $\mathsf{FSK}^d$ to formulas provable in $\mathsf{N4CK}$ and leaves intact the applications of \eqref{E:mp}. Therefore $Tr_\phi(\psi_1),\ldots,Tr_\phi(\psi_n)$ is a deduction of $Tr_\phi(\psi_n) = (Tr_\phi(\chi_1)\vee\ldots\vee Tr_\phi(\chi_m))$ in $\mathsf{N4CK}$ from premises in $Tr_\phi(\Gamma)$.
	
	In the other direction, let $\psi_1,\ldots,\psi_n = (Tr_\phi(\chi_1)\vee\ldots\vee Tr_\phi(\chi_m))$ be a deduction in $\mathsf{N4CK}$ from premises in $Tr_\phi(\Gamma)$. Consider the sequence of $\mathcal{L}_\Box$-formulas 
	$$
	\overline{Tr}(\psi_1),\ldots,\overline{Tr}(\psi_n) = \overline{Tr}(Tr_\phi(\chi_1)\vee\ldots\vee Tr_\phi(\chi_m)) = (\chi_1\vee\ldots\vee \chi_m),
	$$
	 where the last equality holds by the Claim in the proof of Lemma \ref{L:embed}.  By Corollary \ref{C:embed}, the application of $\overline{Tr}$ brings formulas provable in $\mathsf{N4CK}$ to formulas provable in $\mathsf{FSK}^d$ and leaves intact the applications of \eqref{E:mp}. Therefore $\overline{Tr}(\psi_1),\ldots,\overline{Tr}(\psi_n)$ is a deduction of $\overline{Tr}(\psi_n) =(\chi_1\vee\ldots\vee \chi_m)$ in $\mathsf{FSK}^d$ from premises in (again applying the Claim from the proof of  Lemma \ref{L:embed}) $\overline{Tr}(Tr_\phi(\Gamma)) = \Gamma$.
\end{proof}
Finally, set $\mathsf{IK}^{e+}:= \mathsf{L}(\mathcal{L}^{e+}_{(\Box,\Diamond)}, MInt^e, \models_m^i)$. The following proposition is well-known in the existing literature (cf. the proof of \cite[Proposition 7]{odintsovwansing}):
\begin{proposition}\label{P:modal-embed}
	Let the mapping $E^m:\mathcal{L}_\Box\to\mathcal{L}^{e+}_{(\Box,\Diamond)}$ be defined by induction on the construction of $\phi \in \mathcal{L}_\Box$, assuming all the clauses given in the definition of $E$ plus the following ones:
	\begin{align*}
		E^m(\Box\psi):= \Box(E^m(\psi));\\
		E^m(\sim\Box\psi):= \Diamond(E^m(\sim\psi))
	\end{align*}
Then for all $\Gamma, \Delta \subseteq \mathcal{L}_\Box$ it is true that $\Gamma \models_{\mathsf{FSK}^d} \Delta$ iff $E^m(\Gamma) \models_{\mathsf{IK}^{e+}} E^m(\Delta)$.
\end{proposition}
\begin{proof}[Proof (a sketch)]
We argue as in the proof of Proposition \ref{P:prop-embedding}, first establishing the following claims:

\textit{Claim 1}. Let $\mathcal{M} = (W, \leq, R, V^+, V^-)\in MNel$. Then $\mathcal{M}^i = (W, \leq, R, V)\in MInt^e$  where, for an $i\in \omega$,  we set $V(p_i): = V^+(p_i)$ and $V(q_i) := V^-(p_i)$. Moreover, we have $
\mathcal{M}, w \models_m^+ \phi\text{ iff }\mathcal{M}^i, w \models_m^i E^m(\phi)$
for every $w \in W$ and every $\phi \in \mathcal{L}_\Box$.

The proof of Claim 1 is by induction for which we only consider the modal case.

Assume $\phi = \Box\psi$. Then $\mathcal{M}, w \models_m^+ \phi$ iff $(\forall v \geq w)(R(v,u)\text{ implies }\mathcal{M}, u \models_m^+ \psi)$ iff, by IH, $(\forall v \geq w)(R(v,u)\text{ implies }\mathcal{M}, u \models_m^i E^m(\psi))$ iff $\mathcal{M}^i, w \models_m^i E^m(\phi)$.

On the other hand, $\mathcal{M}, w \models_m^+ \sim\phi$ iff $(\exists u \in W)(R(w,u)\text{ and }\mathcal{M}, u \models_m^+ \sim\psi)$ iff, by IH, $(\exists u \in W)(R(w,u)\text{ and }\mathcal{M}^i, u \models_m^i E^m(\sim\psi))$ iff $\mathcal{M}^i, w \models_m^i \Diamond E^m(\sim\psi) = E^m(\sim\phi)$.  

\textit{Claim 2}. Let $\mathcal{M} = (W, \leq, R, V)\in MInt^e$. Then $\mathcal{M}^{n4} = (W, \leq, R, V^+, V^-)\in Nel$,  where, for an $i \in \omega$, we set $V^+(p_i): = V(p_i)$ and $V^-(p_i) := V(q_i)$, and, moreover, we have $\mathcal{M}^{n4}, w \models_m^+ \phi\text{ iff }\mathcal{M}, w \models_m^i E^m(\phi)$ for every $w \in W$ and every $\phi \in \mathcal{L}_\Box$.

The proof is similar to the proof of Claim 1.

Now, if $\mathcal{M}, w \models_{\mathsf{FSK}^d} (\Gamma, \Delta)$, then, by Claim 1, $\mathcal{M}^i, w\models_{\mathsf{IK}^{e+}}(E^m(\Gamma), E^m(\Delta))$. Conversely, if $\mathcal{M}, w\models_{\mathsf{IK}^{e+}}(E^m(\Gamma), E^m(\Delta))$, then, by Claim 2, $\mathcal{M}^{n4}, w\models_{\mathsf{FSK}^d} (\Gamma, \Delta)$.   
\end{proof}
Summing up the results reported in this paper, we get the following cascade of embeddings for every $\phi \in \mathcal{L}_{\boxto}$:
\begin{diagram}
	\mathsf{N4CK} &\rTo^{E^\pm,\,E^\mp} &\mathsf{IntCK}^{e+}\\
	\uTo^{Tr_\phi} &      & \uTo_{Tr^i_\phi}\\
\mathsf{FSK}^d &\rTo_{E^m} & \mathsf{IK}^{e+}
\end{diagram}
Note that the diagram does not commute irrespective of the choice between $E^\pm$ and $E^\mp$. Indeed, we have, for instance $E^\pm(Tr_{p_1}(\Box p_1)) = E^\pm(p_1\boxto p_1) = (p_1 \boxto (q_1 \boxto p_1))$ and $E^\mp(Tr_{p_1}(\Box p_1)) =  (q_1 \boxto (p_1 \boxto p_1))$, but $Tr^i_{p_1}(E^m(\Box p_1)) = Tr^i_{p_1}(\Box p_1) = p_1 \boxto p_1$. Thus the mappings $Tr_\phi\circ E^\pm$, $Tr_\phi\circ E^\mp$, and $E^m\circ Tr^i_\phi$ form a group of three pairwise different faithful embeddings of $\mathsf{FSK}^d$ into $\mathsf{IntCK}^{e+}$ for every $\phi \in \mathcal{L}_{\boxto}$.

Furthermore, it is easy to notice that we have both $Tr_\phi\circ E^+ = E^m\circ Tr^i_\phi$ and $Tr_\phi\circ E^- = E^m\circ Tr^i_{\sim\phi}$; it follows then, that even though $E^+$ and $E^-$ do not provide us with a faithful embedding of 
$\mathsf{N4CK}$ into $\mathsf{IntCK}^{e+}$, they remain faithful when restricted to $Tr_\phi$-images of modal formulas. 

\section{Conclusion, discussion, and future work}\label{S:conclusion}
We have introduced and axiomatized the Nelsonian paraconsistent conditional logic $\mathsf{N4CK}$, and we have looked into its relations to other conditional logics like $\mathsf{CK}$ and $\mathsf{IntCK}$ as well as its relations to the $\mathsf{N4}$-based modal logic $\mathsf{FSK}^d$. In doing so, we mainly focused on different faithful embeddings arising between the logics in question; when constructing these embeddings, we paid special attention to reproducing some of the characteristic properties of $\mathsf{N4}$ which is the reduct of $\mathsf{N4CK}$ to the purely propositional language. Another major source of analogies was provided by $\mathsf{IntCK}$: we used the relation between $\mathsf{IntCK}$ and $\mathsf{IK}$ in our argument that $\mathsf{FSK}^d$ is the correct modal companion for $\mathsf{N4CK}$ and we have also constructed a form of joint (although non-commuting) embedding  of the pair $\{\mathsf{FSK}^d, \mathsf{N4CK}\}$ into the positive fragment of the pair $\{\mathsf{IK}, \mathsf{IntCK}\}$.

However, one element is saliently missing in this web of faithful embeddings, and it is the embedding of $\mathsf{N4CK}$ into $\mathsf{QN4}$, the first-order variant of $\mathsf{N4}$. Indeed, both $\mathsf{IK}$ and $\mathsf{IntCK}$ have been shown to be embeddable into $\mathsf{QInt}$, the first-order intuitionistic logic, and the embedding of $\mathsf{FSK}^d$ into $\mathsf{QN4}$ was shown to be faithful in \cite{odintsovwansing} by a remarkably simple argument which basically fed the negation normal forms arising in $\mathsf{FSK}^d$ into the composition of $E^m$ and the embedding of $\mathsf{IK}$ into $\mathsf{QInt}$.

Whereas this method can also be used in the case of $\mathsf{N4CK}$, it only gives one of the possible faithful embeddings of $\mathsf{N4CK}$ into $\mathsf{QN4}$; yet, other faithful embeddings of this sort are also possible, and one may argue that the embedding arising in this way is neither the most natural nor the most interesting one. The whole question of mapping possible embeddings of $\mathsf{N4CK}$ into $\mathsf{QN4}$, therefore, deserves a separate treatment, and we hope to be able to address it in one of our publications in the nearest future.

On the other hand, it is also interesting to look at other conditional logics based on $\mathsf{N4}$, not just at $\mathsf{N4CK}$. We have already said a thing or two about $\mathsf{N4CK}'$ in this paper, but other promising extensions of $\mathsf{N4CK}$ are clearly possible. We find it especially interesting to look into the extensions of $\mathsf{N4CK}$ which realize different principles of connexive logics (see, e.g. \cite{omw} for an overview) and to see how far can one push the limits of the Nelsonian condional operator in this direction; we hope to do just this in another of our future papers.

Yet another direction for the future research is to look into reproducing the results obtained both in \cite{olkhovikov} and in this paper for other constructive conditional logics, especially the ones based on propositional logics like the original Nelson's logic $\mathsf{N3}$ and the negation-inconsistent connexive logic $\mathsf{C}$, introduced by H. Wansing in \cite{w}.

\medskip
\textbf{Acknowledgements}. This research has received funding from the European Research Council (ERC) under the European Union's Horizon 2020 research and innovation programme, grant agreement ERC-2020-ADG, 101018280, ConLog.

\appendix

\section{Proof of Lemma \ref{L:theorems}}\label{A:1}
We can derive \eqref{E:RCbox} by combining \eqref{E:RCbox1} and \eqref{E:RCbox2}. As for the other items, we sketch the respective proofs and derivations in $\mathbb{N}4\mathbb{CK}$:
\begin{align}
	\eqref{E:Rnec}:\qquad\phi\label{E:p1}& &&\text{premise}\\
	\phi& \leftrightarrow (\phi\to\phi)\label{E:p3}&&\text{\eqref{E:p1}, $\mathsf{N4}$}\\
	(\psi&\boxto\phi)\leftrightarrow(\psi\boxto(\phi\to\phi))\label{E:p4}&&\text{\eqref{E:p3}, \eqref{E:RCbox1}}\\
	\psi&\boxto\phi\label{E:p5}&&\text{\eqref{E:p4}, \eqref{E:a4}, $\mathsf{N4}$}
\end{align}
\begin{align}
	\eqref{E:Rmbox}:\qquad\phi&\to\psi\label{E:p6} &&\text{premise}\\
	(\phi&\wedge\psi)\leftrightarrow \phi\label{E:p7}&&\text{\eqref{E:p6}, $\mathsf{N4}$}\\
	(\chi&\boxto(\phi\wedge\psi))\leftrightarrow(\chi\boxto\phi)\label{E:p8}&&\text{\eqref{E:p7}, \eqref{E:RCbox1}}\\
	(\chi&\boxto\phi)\to(\chi\boxto\psi)\label{E:p10}&&\text{\eqref{E:p8},  \eqref{E:a1}, $\mathsf{N4}$}
\end{align}
\begin{align}
	\eqref{E:Rnbox}:\quad\sim\phi&\to\sim\psi\label{E:q0} &&\text{premise}\\
	\sim(\phi&\wedge\psi)\leftrightarrow \sim\psi\label{E:q1}&&\text{\eqref{E:q0}, $\mathsf{N4}$}\\
	\sim(\chi&\boxto(\phi\wedge\psi))\leftrightarrow\sim(\chi\boxto\psi)\label{E:q2}&&\text{\eqref{E:q1}, \eqref{E:RCbox2}}\\		\sim((\chi&\boxto\phi)\wedge(\chi\boxto\psi))\leftrightarrow\sim(\chi\boxto\psi)\label{E:q3}&&\text{\eqref{E:q2}, \eqref{E:a1}, $\mathsf{N4}$}\\
	(\sim(\chi&\boxto\phi)\vee\sim(\chi\boxto\psi))\leftrightarrow\sim(\chi\boxto\psi)\label{E:q4}&&\text{\eqref{E:q3}, $\mathsf{N4}$}\\
	\sim(\chi&\boxto\phi)\to\sim(\chi\boxto\psi)&&\text{\eqref{E:q4}, $\mathsf{N4}$}
\end{align}
\begin{align}
	\eqref{E:Rmdiam}:\quad\phi&\to\psi\label{E:qa0} &&\text{premise}\\
	\sim\sim\phi&\to\sim\sim\psi\label{E:qa1}&&\text{\eqref{E:qa0}, $\mathsf{N4}$}\\
	(\chi&\diamondto\phi)\to(\chi\diamondto\psi)&&\text{\eqref{E:qa1}, \eqref{E:Rnbox}}
\end{align}
\begin{align}
	\eqref{E:T1}:\qquad(\psi&\wedge(\psi\to\chi))\to\chi\label{E:q5}&&\text{$\mathsf{N4}$}\\
	(\phi&\boxto(\psi\wedge(\psi\to\chi)))\to(\phi\boxto\chi)\label{E:q6}&&\text{\eqref{E:q5}, \eqref{E:Rmbox}}\\
	((\phi&\boxto\psi)\wedge(\phi\boxto(\psi\to\chi)))\to(\phi\boxto\chi)&&\text{\eqref{E:a1}, \eqref{E:q6}, $\mathsf{N4}$}
\end{align}
\begin{align}
\eqref{E:T2}:\,	(\sim(\phi&\boxto\psi)\wedge(\phi\boxto(\sim\psi\to\sim\chi)))\to\notag\\
&\qquad\qquad\to\sim(\phi\boxto(\psi\vee\sim(\sim\psi\to\sim\chi)))\label{E:q7} &&\text{\eqref{E:a2}}\\
	\sim(\psi&\vee\sim(\sim\psi\to\sim\chi))\to\sim\chi\label{E:q8}&&\text{$\mathsf{N4}$}\\
	\sim(\phi&\boxto(\psi\vee\sim(\sim\psi\to\sim\chi)))\to\sim(\phi\boxto\chi)\label{E:q9}&&\text{\eqref{E:q8}, \eqref{E:Rnbox}}\\
	(\sim(\phi&\boxto\psi)\wedge(\phi\boxto(\sim\psi\to\sim\chi)))\to\sim(\phi\boxto\chi)&&\text{\eqref{E:q7}, \eqref{E:q9}, $\mathsf{N4}$}
\end{align}
\begin{align}
	\eqref{E:T3}:\quad(\sim(\phi&\boxto\sim(\psi\to \sim\chi))\wedge(\phi\boxto\psi))\to\notag\\
	&\qquad\qquad\to\sim(\phi\boxto(\sim(\psi\to \sim\chi)\vee\sim\psi))\label{E:r0} &&\text{\eqref{E:a2}}\\
	\sim(\sim&(\psi\to \sim\chi)\vee\sim\psi)\to\sim\chi\label{E:r1}&&\text{$\mathsf{N4}$}\\
	\sim(\phi&\boxto(\sim(\psi\to \sim\chi)\vee\sim\psi))\to \sim(\phi\boxto\chi)\label{E:r2}&&\text{\eqref{E:r1}, \eqref{E:Rnbox}}\\
	(\sim(\phi&\boxto\sim(\psi\to \sim\chi))\wedge(\phi\boxto\psi))\to\sim(\phi\boxto\chi)&&\text{\eqref{E:r0}, \eqref{E:r1}, $\mathsf{N4}$}
\end{align}
\begin{align}
	\eqref{E:T4}:\quad\sim(\phi&\boxto(\sim\psi\wedge \sim\chi))\to\sim((\phi\boxto\sim\psi)\wedge (\phi\boxto\sim\chi))\label{E:ra0}\qquad\text{\eqref{E:a1}}\\
	\sim((\phi&\boxto\sim\psi)\wedge (\phi\boxto\sim\chi))\to((\phi\diamondto\psi)\vee(\phi\diamondto\chi))\label{E:ra1}\quad\,\,\text{\eqref{E:a0.2}}\\
	\sim\sim(&\psi\vee\chi)\to\sim(\sim\psi\wedge\sim\chi)\label{E:ra2}\qquad\qquad\qquad\qquad\qquad\qquad\qquad\text{$\mathsf{N4}$}\\
	(\phi&\diamondto(\psi\vee\chi))\to\sim(\phi\boxto(\sim\psi\wedge \sim\chi))\label{E:ra3}\qquad\qquad\quad\text{\eqref{E:ra2}, \eqref{E:Rnbox}}\\
	(\phi&\diamondto(\psi\vee\chi))\to((\phi\diamondto\psi)\vee(\phi\diamondto\chi))\qquad\qquad\text{\eqref{E:ra0}, \eqref{E:ra1}, \eqref{E:ra3}}
\end{align}	
\begin{align}
	\eqref{E:T5}:\quad\psi&\to(\sim\sim(\psi\to\chi)\to\sim\sim\chi)\label{E:rb0} &&\text{$\mathsf{N4}$}\\
	(\phi&\boxto\psi)\to(\phi\boxto(\sim\sim(\psi\to\chi)\to\sim\sim\chi))\label{E:rb1} &&\text{\eqref{E:rb0}, \eqref{E:Rmbox}}\\
	(\phi&\boxto(\sim\sim(\psi\to\chi)\to\sim\sim\chi))\to\notag\\
	&\qquad\qquad\to((\phi\diamondto(\psi\to\chi))\to (\phi\diamondto\chi))\label{E:rb2} &&\text{\eqref{E:T2}}\\
	(\phi&\boxto\psi)\to((\phi\diamondto(\psi\to\chi))\to (\phi\diamondto\chi)) &&\text{\eqref{E:rb1}, \eqref{E:rb2}, $\mathsf{N4}$}
\end{align}	
\begin{align}
	\eqref{E:T6}:\quad\sim\psi&\leftrightarrow\sim\sim\sim\psi\label{E:rb3} &&\text{\eqref{E:a0.1}}\\
	\sim(\phi&\boxto\psi)\leftrightarrow(\phi\diamondto\sim\psi) &&\text{\eqref{E:rb3}, \eqref{E:RCbox2}}
\end{align}

\section{Proof of Lemma \ref{L:CK}}\label{A:2}
	(Part 1) Note that \eqref{E:a0.1}--\eqref{E:a0.4} are clearly valid in $\mathsf{CL}$. Moreover, \eqref{E:Rmbox} and \eqref{E:T1} can be deduced in $\mathsf{CK}$ in the same way as in $\mathsf{N4CK}$. We sketch the proofs for the remaining axioms and inference rules:
	\begin{align}
		\eqref{E:a2}:\quad&(\phi\boxto\chi)\to(\phi\boxto(\neg(\psi\wedge \chi)\to\neg\psi))\label{E:ck2}&&\text{$\mathsf{CL}$,\eqref{E:Rmbox}}\\
		&(\phi\boxto\chi)\to((\phi\boxto\neg(\psi\wedge \chi))\to(\phi\boxto\neg\psi))\label{E:ck3}&&\text{\eqref{E:ck2},\eqref{E:T1},$\mathsf{CL}$}\\
		&(\phi\boxto\chi)\to(\neg(\phi\boxto\neg\psi)\to\neg(\phi\boxto\neg(\psi\wedge \chi)))\label{E:ck4}&&\text{\eqref{E:ck3},$\mathsf{CL}$}\\
		&(\phi\boxto\chi)\to((\phi\diamondto\psi)\to(\phi\diamondto(\psi\wedge \chi)))&&\text{\eqref{E:ck4},$\mathsf{CL}$}
	\end{align} 
	\begin{align}
		\eqref{E:a3}:\quad&\sim(\phi\boxto\sim\psi)\to(\phi\boxto\chi)\label{E:ck5}&&\text{premise}\\
		&(\phi\boxto\sim\psi)\to(\phi\boxto(\psi\to\chi))\label{E:ck8}&&\text{$\mathsf{CL}$,\eqref{E:Rmbox}}\\
		&(\phi\boxto\chi)\to(\phi\boxto(\psi\to\chi))\label{E:ck9}&&\text{$\mathsf{CL}$,\eqref{E:Rmbox}}\\
		&\sim(\phi\boxto\sim\psi)\to(\phi\boxto(\psi\to\chi))\label{E:ck11}&&\text{\eqref{E:ck5},\eqref{E:ck9},$\mathsf{CL}$}\\
		&((\phi\boxto\sim\psi)\vee\sim(\phi\boxto\sim\psi))\to(\phi\boxto(\psi\to\chi))\label{E:ck12}&&\text{\eqref{E:ck8},\eqref{E:ck11},$\mathsf{CL}$}\\
		&\phi\boxto(\psi\to\chi)\label{E:ck14}&&\text{\eqref{E:ck12},$\mathsf{CL}$}
	\end{align}
	\begin{align}
		\eqref{E:RCbox2}:\quad&\sim\phi\leftrightarrow\sim\psi \in \mathsf{CK}\label{E:kc4}&&\text{premise}\\
			&\phi\leftrightarrow\psi \in \mathsf{CK}\label{E:kc5}&&\text{\eqref{E:kc4}, $\mathsf{CL}$}\\
		&(\chi\boxto\phi)\leftrightarrow(\chi\boxto\psi)\in \mathsf{CK}\label{E:kc6}&&\text{\eqref{E:kc5},\eqref{E:RCbox1}}\\
		&\sim(\chi\boxto\phi)\leftrightarrow\sim(\chi\boxto\psi)\in \mathsf{CK}&&\text{\eqref{E:kc6}, $\mathsf{CL}$}
	\end{align}
	Having now every element of $\mathbb{N}4\mathbb{CK}$ deduced in $\mathsf{CK}$, we can deduce the remaining parts of Lemma \ref{L:theorems} as it was done in Section \ref{sub:axiomatization}.
	
\section{Proof of Lemma \ref{L:theorems-m}}\label{A:3}
\begin{align}
	\eqref{E:rmnbox}:\qquad\sim\phi&\to\sim\psi\label{E:ma2} &&\text{premise}\\
	\Diamond\sim\phi&\to\Diamond\sim\psi\label{E:ma3}&&\text{\eqref{E:ma2}, \eqref{E:rmdiam}}\\
	\sim\Box\phi&\to\sim\Box\psi&&\text{\eqref{E:ma3}, \eqref{E:aa6}}
\end{align}
\begin{align}
	\eqref{E:t2}:\qquad\Box(&\phi\wedge\psi)\to(\Box\phi\wedge\Box\psi)\label{E:ma4} &&\text{\eqref{E:rmbox}, $\mathsf{N4}$}\\
	\sim\Box&\phi\to\sim\Box(\phi\wedge\psi)\label{E:ma5}&&\text{\eqref{E:rmnbox}, $\mathsf{N4}$}\\
	\sim\Box&\psi\to\sim\Box(\phi\wedge\psi)\label{E:ma6}&&\text{\eqref{E:rmnbox}, $\mathsf{N4}$}\\
	\sim(\Box&\phi\wedge\Box\psi)\to\sim\Box(\phi\wedge\psi)\label{E:ma7}&&\text{\eqref{E:ma5}, \eqref{E:ma6}, $\mathsf{N4}$}\\
	\Diamond(\sim&\phi\vee\sim\psi)\to(\Diamond\sim\phi\vee\Diamond\sim\psi)\label{E:ma8}&&\text{\eqref{E:aa3}}\\
	\sim(\phi&\wedge\psi)\leftrightarrow\sim\sim(\sim\phi\vee\sim\psi)\label{E:ma9} &&\text{$\mathsf{N4}$}\\
	\sim\Box(&\phi\wedge\psi)\leftrightarrow\Diamond(\sim\phi\vee\sim\psi)\label{E:ma10} &&\text{\eqref{E:ma9}, \eqref{E:rnbox}}\\
	(\Diamond\sim\phi&\vee\Diamond\sim\psi)\to(\sim\Box\phi\vee\sim\Box\psi)\label{E:ma11} &&\text{\eqref{E:aa6}, $\mathsf{N4}$}\\
	(\Diamond\sim\phi&\vee\Diamond\sim\psi)\to\sim(\Box\phi\wedge\Box\psi)\label{E:ma12} &&\text{\eqref{E:ma11}, $\mathsf{N4}$}\\
	\sim\Box(&\phi\wedge\psi)\to\sim(\Box\phi\wedge\Box\psi)\label{E:ma13} &&\text{\eqref{E:ma8}, \eqref{E:ma10}, \eqref{E:ma12}, $\mathsf{N4}$}\\
	\Box(\phi&\wedge\psi)\Leftrightarrow(\Box\phi\wedge\Box\psi)\label{E:ma14} &&\text{\eqref{E:aa1}, \eqref{E:ma4}, \eqref{E:ma7}, \eqref{E:ma13}, $\mathsf{N4}$}
\end{align}
\begin{align}
	\eqref{E:t3}:\qquad(\sim\Box&\phi\wedge\Box\psi)\to(\Diamond\sim\phi\wedge\Box\psi)\label{E:mb1} &&\text{\eqref{E:aa6}, $\mathsf{N4}$}\\
	\sim&\phi\to(\psi\to\sim(\phi\vee\sim\psi))\label{E:mb2}&&\text{$\mathsf{N4}$}\\
	\Diamond\sim&\phi\to\Diamond(\psi\to\sim(\phi\vee\sim\psi))\label{E:mb3}&&\text{\eqref{E:mb2}, \eqref{E:rmdiam}}\\
	(\Diamond\sim&\phi\wedge\Box\psi)\to (\Diamond(\psi\to\sim(\phi\vee\sim\psi))\wedge\Box\psi)\label{E:mb4}&&\text{\eqref{E:mb3}, $\mathsf{N4}$}\\
	(\Diamond\sim&\phi\wedge\Box\psi)\to\Diamond\sim(\phi\vee\sim\psi)\label{E:mb5}&&\text{\eqref{E:aa4}, \eqref{E:mb4}, $\mathsf{N4}$}\\
	\Diamond\sim(&\phi\vee\sim\psi)\to\sim\Box(\phi\vee\sim\psi)\label{E:mb6}&&\text{\eqref{E:aa6}}\\
	(\sim\Box&\phi\wedge\Box\psi)\to\sim\Box(\phi\vee\sim\psi)&&\text{\eqref{E:mb1}, \eqref{E:mb5}, \eqref{E:mb6}}
\end{align}
}
\end{document}